\documentclass[11pt,reqno]{amsart}
\usepackage[usenames]{color}
\usepackage{amsmath, amssymb, latexsym, multicol, amsthm, color, enumitem,comment, tikz-cd}
\usepackage[all]{xy}
\usepackage{hyperref}
\usepackage{cleveref}
\newtheoremstyle{natural}
{\parsep}   
{\parsep}   
{\normalfont}  
{0pt}       
{\bfseries} 
{.}         
{5pt plus 1pt minus 1pt} 
{}          

\newtheoremstyle{remark}
{\parsep}   
{\parsep}   
{\normalfont}  
{0pt}       
{\itshape} 
{.}         
{5pt plus 1pt minus 1pt} 
{}          
\theoremstyle{natural}
\newtheorem{thm}{Theorem}[section]
\newtheorem{cor}[thm]{Corollary}
\newtheorem{lem}[thm]{Lemma}
\newtheorem{prp}[thm]{Proposition}
\theoremstyle{remark}
\newtheorem{rmk}[thm]{Remark}
\numberwithin{equation}{section}

\crefname{thm}{theorem}{theorems}
\crefname{prp}{proposition}{propositions}
\crefname{cor}{corollary}{corollaries}

\newcommand{\Z}{\mathbb{Z}}
\newcommand{\Q}{\mathbb{Q}}
\newcommand{\N}{\mathbb{N}}
\newcommand{\R}{\mathbb{R}}
\newcommand{\C}{\mathbb{C}}
\newcommand{\A}{\mathbb{A}}
\newcommand{\bs}{\backslash}
\newcommand{\cb}[1]{\left\{{#1}\right\}}
\newcommand{\cbm}[2]{\left\{{#1}\;\middle|\;{#2}\right\}}
\newcommand{\diag}{\operatorname{diag}}
\newcommand{\e}{\operatorname{e}}
\newcommand{\id}{\operatorname{id}}
\newcommand{\im}{\operatorname{Im}}

\newcommand{\pd}[2]{\frac{\partial{#1}}{\partial{#2}}}
\newcommand{\ppmod}[1]{\hspace{-0.2cm}\pmod{#1}}
\newcommand{\ol}[1]{\overline{#1}}
\newcommand{\ord}{\operatorname{ord}}
\newcommand{\rb}[1]{\left({#1}\right)}

\newcommand{\sbe}{\subseteq}

\newcommand{\vb}[1]{\left| {#1} \right|}

\newcommand{\Res}{\operatorname{Res}}
\newcommand{\GL}{\operatorname{GL}}
\newcommand{\SL}{\operatorname{SL}}
\newcommand{\Sp}{\operatorname{Sp}}
\newcommand{\fsp}{\mathfrak{sp}}
\newcommand{\mc}{\mathcal}

\newcommand{\ba}{\[\begin{aligned}}
\newcommand{\ea}{\end{aligned}\]}
\newcommand{\bp}{\begin{pmatrix}}
\newcommand{\ep}{\end{pmatrix}}
\newcommand{\bv}{\begin{vmatrix}}
\newcommand{\ev}{\end{vmatrix}}
\newcommand{\bea}{\begin{enumerate}[label=(\alph*)]}
\newcommand{\ber}{\begin{enumerate}[label=(\roman*)]}
\newcommand{\ben}{\begin{enumerate}[label=(\arabic*)]}
\newcommand{\beani}{\begin{enumerate}[leftmargin=*, label=(\alph*), wide, labelindent=0pt, labelwidth=*]}
\newcommand{\berni}{\begin{enumerate}[leftmargin=*, label=(\roman*), wide, labelindent=0pt, labelwidth=*]}
\newcommand{\benni}{\begin{enumerate}[leftmargin=*, label=(\arabic*), wide, labelindent=0pt, labelwidth=*]}
\newcommand{\ee}{\end{enumerate}}

\title{Fourier coefficients of $\Sp(4)$ Eisenstein Series}

\author{Siu Hang Man}
\address{Charles University, Faculty of Mathematics and Physics, Department of Algebra, Sokolovská 49/83, 186 75 Praha 8, Czechia}
\email{shman@karlin.mff.cuni.cz}

\thanks{The author is supported by the DAAD Graduate School Scholarship Programme, and OP RDE project No. CZ.02.2.69/0.0/0.0/18\_053/0016976 International mobility of research, technical and administrative staff at the Charles University.}

\makeatletter
\@namedef{subjclassname@2020}{%
  \textup{2020} Mathematics Subject Classification}
\makeatother

\begin{document}

\begin{abstract}
We compute explicit formulae for the constant terms and Fourier coefficients for Eisenstein series on $\Sp(4,\R)$, in terms of zeta functions and Whittaker functions. We also develop a generalisation of Ramanujan sums to $\Sp(4,\Z)$, which appears as coefficients in the Fourier expansion for the minimal Eisenstein series.
\end{abstract}

\maketitle

\section{Introduction}

In the theory of automorphic forms, Eisenstein series are important building blocks of the spectral decomposition. The goal of the article is to give a very explicit formulation of properties of $\Sp(4)$ Eisenstein series in the classical language. Much of the theory was already worked out implicitly in the works of Langlands \cite{Arthur1979,Langlands1976}. However, applications in analytic number theory often require explicit formulae. This applies in particular to trace formulae and relative trace formulae (à la Kuznetsov) whose use in analytic number theory is based on their explicit shapes \cite{Blomer2021}. Such formulae are only worked out for few groups. Besides the classical case $\GL(2)$, such explicit computations have only been done for $\GL(3)$ by Bump, Goldfeld and others \cite{BalakciThesis,Bump1984,BFG1988,Goldfeld2006,ThillainatesanThesis}, with hints on how to generalise to $\GL(n)$. The group $\Sp(4)$ is a natural candidate as a first step for the generalisation of these computations to a group other than $\GL(n)$. It is worth noting that some work has been done on the exceptional group $G_2$ \cite{JR1997,Xiong2017}.

Eisenstein series find many applications in number theory. The Fourier coefficients of Eisenstein series feature in the construction of automorphic L-functions by the Langlands-Shahidi method \cite{Shahidi2010}. Eisenstein series are also connected with algebraic objects, such as quadratic forms \cite{Blomer2020,Siegel1939} and algebraic varieties \cite{FMT1989}. Through the construction of the Eisenstein series, we see that their Fourier coefficients feature in a generalised version of exponential sums and divisor-type functions, which are worthy of investigating in their own right. 

Let $G = \Sp(4,\R)$. We shall fix a maximal compact subgroup $K$ and a Borel subgroup $B$ of $G$ (see \Cref{section:group_decomp}). Let $\Gamma = \Sp(4,\Z)$, the standard arithmetic subgroup of $G$. Let $P_0, P_\alpha, P_\beta$ denote respectively the standard minimal, Siegel and Jacobi parabolic subgroups, with respect to the chosen Borel subgroup $B$. Then we have the Levi decompositions $P_j = N_j M_j$, $j\in\cb{0,\alpha,\beta}$ of the parabolic subgroups, where $N_j$ is unipotent and $M_j$ is the Levi subgroup. The Eisenstein series for the minimal parabolic subgroup $P_0$ is then defined as a function on $G/K$ to be
\ba
E_0(g, \nu) := \sum\limits_{\gamma \in (P_0 \cap \Gamma) \bs \Gamma} I_0(\gamma g, \nu),
\ea
for $g\in G/K$ of the form \eqref{eq:g_iwasawa}, where $\nu = (\nu_1, \nu_2) \in \C^2$, and $I_0(g, \nu) := y_1^{\nu_1+2} y_2^{2\nu_2-\nu_1+1}$. 

Let $f$ be an automorphic form on $\GL(2)$. Denote by $m_\alpha: G/K\to M_\alpha/ (K\cap M_\alpha)$ the projection map with respect to the decomposition $G = P_\alpha K = N_\alpha M_\alpha K$. The map $m_\beta: G/K \to M_\beta/ (K\cap M_\beta)$ is defined analogously. The Eisenstein series for the Siegel parabolic subgroup $P_\alpha$ is defined to be
\ba
E_\alpha (g, \nu,f) := \sum\limits_{\gamma \in (P_\alpha \cap \Gamma) \bs \Gamma} f(m_\alpha(\gamma g)) I_\alpha(\gamma g, \nu),
\ea
where $\nu\in \C$, and $I_\alpha(g,\nu) := y_1^{\nu+3/2} y_2^{\nu+3/2}$. The Eisenstein series for the Jacobi parabolic subgroup $P_\beta$ is defined to be
\ba
E_\beta (g,\nu,f) := \sum\limits_{\gamma \in (P_\beta \cap \Gamma) \bs \Gamma} f(m_\beta(\gamma g)) I_\beta(\gamma g, \nu),
\ea
where $\nu\in\C$, and $I_\beta(g,\nu) := y_1^{\nu+2}$. 

It is well-known (see \cite{Langlands1976, MW1995}) that these Eisenstein series, while originally defined on an open subset of the complex space where the series converge, can be extended meromorphically to functions defined on the whole complex space.

Let $U$ be the maximal unipotent subgroup of $G$, with respect to the chosen Borel subgroup $B$. Let $\chi = \chi_{m_1,m_2}$ be a character of $U(\Z)\bs U(\R)$ (see \eqref{eq:cd}). Then the Fourier coefficient for an Eisenstein series $E$ (or actually any automorphic function) corresponding to $\chi$ is given by (see \cite{Shahidi2010})
\ba
E_\chi (g) := \int_{U(\Z)\bs U(\R)} E(\eta g) \bar\chi(\eta) d\eta.
\ea

Our main result is a completely explicit formula for the Fourier coefficients of the minimal Eisenstein series $E_0(g,\nu)$.

\begin{thm}\label{Fourier:min}
Let $\chi = \chi_{m_1,m_2}$ be a character of $U(\Z)\bs U(\R)$. Then the Fourier coefficients of the minimal Eisenstein series $E_0 (g,\nu)$ are given as follows. For $m_1=m_2=0$ we have
\ba
E_{0, \chi_{0,0}} (g,\nu) &= W_{\id} (g,\nu,\chi_{0,0}) + \frac{\zeta(2\nu_1-2\nu_2)}{\zeta(2\nu_1-2\nu_2+1)} W_{s_\alpha} (g,\nu,\chi_{0,0}) + \frac{\zeta(2\nu_2-\nu_1)}{\zeta(2\nu_2-\nu_1+1)} W_{s_\beta} (g,\nu,\chi_{0,0})\\
&\hspace{0.5cm}+ \frac{\zeta(2\nu_1-2\nu_2)}{\zeta(2\nu_1-2\nu_2+1)}\frac{\zeta(\nu_1)}{\zeta(\nu_1+1)} W_{s_\alpha s_\beta} (g,\nu,\chi_{0,0})\\
&\hspace{0.5cm}+ \frac{\zeta(2\nu_2-\nu_1)}{\zeta(2\nu_2-\nu_1+1)} \frac{\zeta(2\nu_2)}{\zeta(2\nu_2+1)} W_{s_\beta s_\alpha} (g,\nu,\chi_{0,0})\\
&\hspace{0.5cm}+\frac{\zeta(\nu_1)}{\zeta(\nu_1+1)} \frac{\zeta(2\nu_1-2\nu_2)}{\zeta(2\nu_1-2\nu_2+1)} \frac{\zeta(2\nu_2)}{\zeta(2\nu_2+1)} W_{s_\alpha s_\beta s_\alpha} (g,\nu,\chi_{0,0})\\
&\hspace{0.5cm}+\frac{\zeta(2\nu_2)}{\zeta(2\nu_2+1)} \frac{\zeta(2\nu_2-\nu_1)}{\zeta(2\nu_2-\nu_1+1)} \frac{\zeta(\nu_1)}{\zeta(\nu_1+1)} W_{s_\beta s_\alpha s_\beta} (g,\nu,\chi_{0,0})\\
&\hspace{0.5cm}+\frac{\zeta(2\nu_1-2\nu_2)}{\zeta(2\nu_1-2\nu_2+1)}\frac{\zeta(2\nu_2-\nu_1)}{\zeta(2\nu_2-\nu_1+1)}\frac{\zeta(\nu_1)}{\zeta(\nu_1+1)}\frac{\zeta(2\nu_2)}{\zeta(2\nu_2+1)} W_{w_0} (g,\nu,\chi_{0,0}).
\ea
For $m_1\ne 0$, $m_2=0$ we have
\ba
E_{0, \chi_{m_1, 0}} (g,\nu) &= \frac{\sigma_{2\nu_2-2\nu_1}(m_1)}{\zeta(2\nu_1-2\nu_2+1)} W_{s_\alpha} (g,\nu,\chi_{m_1,0}) + \frac{\zeta(2\nu_2-\nu_1)}{\zeta(2\nu_2-\nu_1+1)} \frac{\sigma_{-2\nu_2}(m_1)}{\zeta(2\nu_2+1)} W_{s_\beta s_\alpha}(g,\nu,\chi_{m_1,0})\\
&\hspace{0.5cm}+\frac{\zeta(\nu_1)}{\zeta(\nu_1+1)} \frac{\zeta(2\nu_1-2\nu_2)}{\zeta(2\nu_1-2\nu_2+1)} \frac{\sigma_{-2\nu_2}(m_1)}{\zeta(2\nu_2+1)} W_{s_\alpha s_\beta s_\alpha}(g,\nu,\chi_{m_1,0})\\
&\hspace{0.5cm}+ \frac{\sigma_{2\nu_2-2\nu_1} (m_1)}{\zeta(2\nu_1-2\nu_2+1)}\frac{\zeta(2\nu_2-\nu_1)}{\zeta(2\nu_2-\nu_1+1)}\frac{\zeta(\nu_1)}{\zeta(\nu_1+1)}\frac{\zeta(2\nu_2)}{\zeta(2\nu_2+1)} W_{w_0}(g,\nu,\chi_{m_1, 0}).
\ea
For $m_1=0$, $m_2\ne 0$ we have
\ba
E_{0, \chi_{0,m_2}} (g,\nu) &= \frac{\sigma_{\nu_1-2\nu_2}(m_2)}{\zeta(2\nu_2-\nu_1+1)} W_{s_\beta} (g,\nu,\chi_{0,m_2}) + \frac{\zeta(2\nu_1-2\nu_2)}{\zeta(2\nu_1-2\nu_2+1)} \frac{\sigma_{-\nu_1}(m_2)}{\zeta(\nu_1+1)} W_{s_\alpha s_\beta} (g,\nu,\chi_{0,m_2})\\
&\hspace{0.5cm}+ \frac{\zeta(2\nu_2)}{\zeta(2\nu_2+1)} \frac{\zeta(2\nu_2-\nu_1)}{\zeta(2\nu_2-\nu_1+1)}\frac{\sigma_{-\nu_1}(m_2)}{\zeta(\nu_1+1)} W_{s_\beta s_\alpha s_\beta} (g,\nu,\chi_{0,m_2})\\
&\hspace{0.5cm}+ \frac{\sigma_{\nu_1-2\nu_2} (m_2)}{\zeta(2\nu_2-\nu_1+1)}\frac{\zeta(2\nu_1-2\nu_2)}{\zeta(2\nu_1-2\nu_2+1)}\frac{\zeta(\nu_1)}{\zeta(\nu_1+1)}\frac{\zeta(2\nu_2)}{\zeta(2\nu_2+1)} W_{w_0} (g,\nu,\chi_{0, m_2}).
\ea
For $m_1,m_2\ne 0$ we have
\ba
E_{0, \chi_{m_1, m_2}} (g,\nu) = \frac{\sigma_{-\nu_2, \nu_2-\nu_1} (m_1, m_2)}{\zeta(2\nu_1-2\nu_2+1) \zeta(2\nu_2-\nu_1+1) \zeta(\nu_1+1) \zeta(2\nu_2+1)} W_{w_0} (g,\nu,\chi_{m_1, m_2}).
\ea
Here $W_w(g,\nu,\chi)$ are Jacquet's Whittaker functions, defined in \eqref{eq:wd}, $\sigma_\nu(m) = \sum_{d\mid m} d^\nu$ is the divisor sum function, and $\sigma_{\nu_1,\nu_2}(m_1,m_2)$ is a multiplicative function defined in \eqref{eq:sd}.
\end{thm}

This article is organised as follows, and we point out a number of auxiliary results of independent interest that are derived in the rest of the paper. In \Cref{section:group_decomp}, we give for $G = \Sp(4)$ the explicit characterisations of the parabolic subgroups and other associated groups, which are used later on. In \Cref{section:coset_rep}, we describe the coset representatives of $(P_j\cap \Gamma) \bs \Gamma$, $j\in\cb{0,\alpha,\beta}$, using Plücker coodinates as defined in \cite{BFH1990}. We compute the explicit Bruhat decomposition in \Cref{section:Bruhat}. Alternative expressions for Eisenstein series are given in \Cref{Eisenstein_alt}. 

In \Cref{section:const_terms}, we consider the constant terms of the minimal Eisenstein series along different parabolic subgroups. The strategy is to utilise the functional equations of intertwining operators. Explicit formulae for constant terms are given in \Cref{constterm:minmin,constterm:minSiegel,constterm:minNS}. They work out in detail the group-theoretic expression given in \cite[II.1.7]{MW1995}. A particular advantage of this approach is that we obtain a nicer expression for the Fourier coefficient $E_{0,\chi_{0,0}}$ in terms of pure zeta quotients.

In \Cref{section:Sp4_Ramanujan}, we consider a generalisation of Ramanujan sums to $\Sp(4,\Z)$, which appears in the Fourier coefficients of Eisenstein series. The construction of such exponential sums for the $\GL(3)$ case is due to Bump \cite{Bump1984}. As in the $\GL(3)$ case, the degenerate sums reduce to classical Ramanujan sums, justifying the use of the term ``generalisation''. The Dirichlet series associated to this Ramanujan sum is computed in \Cref{prp:Sp4rs}. 

In \Cref{section:Fourier_coeff}, we introduce the Whittaker functions on $\Sp(4)$ in terms of Jacquet integrals \cite{Jacquet1967}. We are then able to compute the Fourier coefficients of the minimal Eisenstein series in terms of these generalised Ramanujan sums and Whittaker functions, proving \Cref{Fourier:min}. 

In \Cref{section:degenerate}, we give results for the residual Eisenstein series $E_\alpha(g,\nu,1)$ and $E_\beta(g,\nu,1)$, by considering them as residues of the minimal Eisenstein series. Formulae for constant terms are given in \Crefrange{constterm:Siegelmin}{constterm:NSNS}, and formulae for Fourier coefficients are given in \Cref{Fourier:Siegel,Fourier:NS}.

\section*{Acknowledgement}

The author would like to thank Valentin Blomer for his guidance on the project.

\section{Group decompositions}\label{section:group_decomp}

Let $G = \Sp(4,\R)$ be the real symplectic group of degree 2, namely
\ba
G = \cbm{g\in\GL(4,\R)}{g^\top\bp 0&I_2\\-I_2&0\ep g = \bp0&I_2\\-I_2&0\ep}, \quad I_2 = \bp 1&0\\0&1\ep,
\ea
where $g^\top$ denotes the matrix transpose of $g$ as usual. Let $T$ and $U$ be a maximal split torus and a maximal unipotent subgroup of $G$ respectively, defined as follows:
\ba
T &= \cb{\diag(y_1,y_2,y_1^{-1},y_2^{-1}) \in G},\\
U &= \cbm{ \bp 1&n_1&n_2&n_3\\&1&n_4&n_5\\&&1\\&&-n_1&1\ep \in G}{ n_3 = n_1n_5+n_4}.
\ea
Then $B:=TU$ is a Borel subgroup of $G$. We shall denote by $T^+$ the subgroup of $T$ with positive entries.

Let $X(T)$ and $X^*(T)$ be the character group and the cocharacter group of $T$ respectively, with the natural pairing $\langle-,-\rangle: X(T)\times X^*(T)\mapsto \Z$. Let $\alpha, \beta\in X(T)$ be such that 
\ba
\alpha(\diag(y_1,y_2,y_1^{-1},y_2^{-1})) &= y_1y_2^{-1}, & \quad \beta(\diag(y_1,y_2,y_1^{-1},y_2^{-1})) &= y_2^2.
\ea
Then $\Delta = \cb{\alpha,\beta}$ is a set of simple roots, and $\Psi^+ = \cb{\alpha, \beta, \alpha+\beta, 2\alpha+\beta}$ is the set of positive roots with respect to $(B,T)$. We denote by $s_\alpha$ and $s_\beta$ the simple reflections in the hyperplane orthogonal to $\alpha$ and $\beta$ respectively. Then the Weyl group is given by
\ba
W = \cb{1,s_\alpha,s_\beta,s_\alpha s_\beta, s_\beta s_\alpha, s_\alpha s_\beta s_\alpha, s_\beta s_\alpha s_\beta, s_\alpha s_\beta s_\alpha s_\beta =: w_0}.
\ea 
The elements of the Weyl group can be embedded in $\Sp(4,\Z)$ by setting
\ba
s_\alpha = \bp & 1\\ 1\\ &&&1\\ &&1\ep, \quad s_\beta = \bp 1\\ &&&1\\ &&1\\ &-1\ep.
\ea 
The standard parabolic subgroups of $G$ are given by 
\ba
P_0 &= \cb{\bp *&*&*&*\\  &*&*&*\\ &&*\\ &&*&* \ep} \cap G & &\text{(minimal parabolic subgroup),}\\
P_\alpha &= \cb{\bp *&*&*&*\\ *&*&*&*\\ &&*&*\\ &&*&* \ep} \cap G & &\text{(Siegel parabolic subgroup),}\\
P_\beta &= \cb{\bp *&*&*&*\\  &*&*&*\\ &&*\\ &*&*&* \ep} \cap G & &\text{(Jacobi parabolic subgroup),}
\ea
corresponding to subsystems of roots generated by $\emptyset, \cb{\alpha}$ and $\cb{\beta}$ respectively. We have Levi decompositions $P_j = N_j M_j$, $j \in \cb{0,\alpha,\beta}$, given by
\ba
N_0 &= \cbm{\bp 1&n_1&n_2&n_3\\&1&n_4&n_5\\&&1\\&&-n_1&1\ep}{n_3=n_1n_5+n_4},\\
M_0 &= \cbm{\diag(y_1, y_2, y_1^{-1}, y_2^{-1}) \in G}{y_1,y_2\in\R^\times},\\
N_\alpha &= \cbm{\bp I_2 & S\\&I_2\ep}{S^\top = S},\\
M_\alpha &= \cbm{\bp A\\&(A^{-1})^\top\ep}{A \in \GL_2(\R)},\\
N_\beta &= \cb{\bp 1&n_1&n_2&n_3\\&1&n_3\\&&1\\&&-n_1&1\ep},\\
M_\beta &= \cbm{\bp y_1\\&a&&b\\&&y_1^{-1}\\&c&&d\ep}{y_1\in\R^\times, \bp a&b\\c&d\ep \in \SL_2(\R)}.
\ea

Let $K$ be the standard maximal compact subgroup of $G$, given by
\ba
K = \cbm{\bp A&B\\-B&A\ep}{A+Bi \in U(2)}.
\ea

By the standard Iwasawa decomposition, elements in $G/K$ can be represented by matrices of the form
\begin{align}\label{eq:g_iwasawa} 
g = \bp 1&n_1&n_2&n_3\\&1&n_4&n_5\\&&1\\&&-n_1&1\ep \bp y_1\\&y_2\\&&y_1^{-1}\\&&&y_2^{-1}\ep,
\end{align}
with $n_3=n_1n_5+n_4$. We may also assume that $y_1, y_2$ are positive. Throughout, we shall fix a Haar measure on $G$. The Haar measure on $K$ will be normalised so that the volume of $K$ is $1$. If $P=NM$ is a parabolic subgroup, then the Haar measure on $P$ is normalised so that it is compatible with the decomposition $G = PK$.

\section{Coset representatives}\label{section:coset_rep}

In this section, we compute the coset representatives of $P_j\cap\Gamma\bs\Gamma$, $j\in\cb{0,\alpha,\beta}$, and compute an explicit Bruhat decomposition, which will be used later on.

\subsection{Plücker coordinates}
Let $G=\Sp(4,\R)$ and $\Gamma = \Sp(4,\Z)$. Let $P_0$ be the standard minimal parabolic subgroup of $G$. We denote by $U_0 = U\sbe P_0$ the unipotent matrices, and $\Gamma_0  = U \cap \Gamma$. We also define
\ba
U_\alpha &:= \cbm{\bp X&Y\\ & (X^{-1})^\top\ep\in G}{X\in\SL(2,\R)}, & \quad \Gamma_\alpha &:= U_\alpha \cap \Gamma,\\
U_\beta &:= \cb{\bp 1&*&*&*\\&*&*&*\\&&1\\&*&*&*\ep \in G}, & \Gamma_\beta &:= U_\beta \cap \Gamma.
\ea

Let
\ba
g = \bp g_{11} & g_{12} & g_{13} & g_{14}\\ g_{21} & g_{22} & g_{23} & g_{24}\\ g_{31} & g_{32} & g_{33} & g_{34}\\ g_{41} & g_{42} & g_{43} & g_{44}\ep \in G.
\ea
We define the following quantities, known as the \emph{Plücker coordinates}, associated to $g\in G$:
\ba
v_i&= g_{3i}, & \quad 1&\le i \le 4,\\
v_{ij} &= g_{3i}g_{4j} - g_{3j}g_{4i}, & \quad 1&\le i < j \le 4.
\ea
It is easy to verify that these quantities are invariant under left action by $\Gamma_0$. The following relations come immediately from the definition:
\begin{equation}\label{eq:Plucker4} 
\begin{aligned}
v_1v_{23}-v_2v_{13}+v_3v_{12} &= 0, & v_1v_{24}-v_2v_{14}+v_4v_{12} &= 0,\\
v_1v_{34}-v_3v_{14}+v_4v_{13} &= 0, & v_2v_{34}-v_3v_{24}+v_4v_{23} &= 0.
\end{aligned}
\end{equation}
Moreover, $G$ being symplectic implies
\begin{equation}\label{eq:Plucker_symp} 
v_{13} + v_{24} = 0.
\end{equation}
Define
\begin{align}\label{V0_def} 
V_0 &= \cbm{v = (v_1, \cdots, v_{34}) \in \R^{10}}{v \text{ satisfies \eqref{eq:Plucker4} and \eqref{eq:Plucker_symp}}},\\
V_\alpha &= \cbm{v = (v_{12}, \cdots, v_{34})\in\R^6}{\begin{array}{l} v_{12}v_{34} - v_{24}v_{13} + v_{14}v_{23} = 0,\\ v_{13}+v_{24}=0.\end{array}},\label{Va_def}\\ 
V_\beta &= \cb{v = (v_1, v_2, v_3, v_4) \in \R^4}. \label{Vb_def} 
\end{align}

From \cite{BFH1990}, we immediately have the follow results.
\begin{prp}
\ben
\item The Plücker coordinates give bijections
\ba
U_0\bs G &\overset\sim\to V_0\bs\cb{0}, \quad & U_\alpha\bs G &\overset\sim\to V_\alpha \bs\cb{0}, \quad & U_\beta\bs G &\overset\sim\to V_\beta \bs\cb{0}.
\ea
\item An orbit of $U_0\bs G$ contains an element of $\Gamma$ if and only if its corresponding Plücker coordinates are such that $(v_1, \cdots, v_4)$ are coprime integers, and $(v_{12}, \cdots v_{34})$ are coprime integers.
\item An orbit of $U_\alpha\bs G$ contains an element of $\Gamma$ if and only if its corresponding Plücker coordinates $(v_{12}, \cdots, v_{34})$ are coprime integers.
\item An orbit of $U_\beta\bs G$ contains an element of $\Gamma$ if and only if its corresponding Plücker coordinates $(v_1, \cdots, v_4)$ are coprime integers.
\ee
\end{prp}

\subsection{Bruhat decomposition}\label{section:Bruhat}
Bruhat decomposition of $G$ is
\ba
G = \coprod\limits_{w\in W} G_w := \coprod\limits_{w\in W} UwTU.
\ea
Hence a coset $\gamma \in U\bs G$ can be represented by a matrix in $wTU = wP_0$ for some $w\in W$; such a Weyl element is unique, and depends on the corresponding Plücker coordinates of the coset. 

We define an equivalence of Plücker coordinates 
\ba
(v_1, v_2, v_3, v_4; v_{12}, v_{13}, v_{14}, v_{23}, v_{24}, v_{34})
\ea
by
\ba
(v_1, \cdots, v_4; v_{12}, \cdots, v_{34}) \sim (k_1 v_1, \cdots, k_1 v_4; k_2 v_{12}, \cdots, k_2 v_{34}) \text{ for } k_1, k_2 \in \R^\times.
\ea

Now we give representatives of $\gamma \in U\bs G$ with the corresponding Plücker coordinates $v$, classified by the Weyl element $w\in W$:

\ben
\item $w = \id$: This says $v \sim (0,0,1,0; 0,0,0,0,0,1)$. In this case, the matrix
\ba
\bp 1/v_3\\ & v_3/v_{34}\\ && v_3\\ &&& v_{34}/v_3\ep = \bp 1\\ &1\\ &&1\\ &&&1\ep  \bp 1/v_3\\ & v_3/v_{34}\\ && v_3\\ &&& v_{34}/v_3\ep
\ea
has the given invariants.

\item $w = s_\alpha$: This says $v \sim (0,0,*,1; 0,0,0,0,0,1)$. Then the matrix
\ba
\bp &1/v_4\\-v_4/v_{34}&v_3/v_{34}\\&&v_3&v_4\\&&-v_{34}/v_4\ep = \bp &1\\1\\&&&1\\&&1\ep \bp -v_4/v_{34} & v_3/v_{34}\\ & 1/v_4\\ && -v_{34}/v_4\\ && v_3 & v_4\ep
\ea
has the given invariants.

\item $w = s_\beta$: This says $v\sim(0,0,1,0; 0,0,0,1,0,*)$. Then the matrix
\ba
\bp 1/v_3\\&&&v_3/v_{23}\\&&v_3\\&-v_{23}/v_3&&v_{34}/v_3\ep = \bp 1\\ &&&1\\ &&1\\ &-1\ep \bp 1/v_3\\ & v_{23}/v_3 && -v_{34}/v_3\\ &&v_3\\ &&& v_3/v_{23}\ep
\ea
has the given invariants.

\item $w = s_\alpha s_\beta$: This says $v \sim (0,1,*,*; 0,0,0,1,0,*)$. Then the matrix
\ba
\bp &&&-1/v_2\\v_2/v_{23}&&&v_3/v_{23}\\&v_2&v_3&v_4\\&&v_{23}/v_2\ep = \bp &&&1\\ 1\\ &-1\\ &&1\ep \bp v_2/v_{23} &&& v_3/v_{23}\\ & -v_2 & -v_3 & -v_4\\ && v_{23}/v_2\\ &&& -1/v_2\ep
\ea
has the given invariants.

\item $w = s_\beta s_\alpha$: This says $v \sim (0,0,*,1; 0,*,1,*,*,*)$. Then the matrix
\ba
\bp &1/v_4\\&&v_4/v_{14}\\&&v_3&v_4\\-v_{14}/v_4&-v_{24}/v_4&-v_{34}/v_4\ep = \bp & 1\\ && 1\\ &&&1\\ -1\ep \bp v_{14}/v_4 & v_{24}/v_4 & v_{34}/v_4\\ & 1/v_4\\ && v_4/v_{14}\\ && v_3 & v_4 \ep
\ea
has the given invariants.

\item $w = s_\alpha s_\beta s_\alpha$: This says $v \sim (1,*,*,*; 0,*,1,*,*,*)$. Then the matrix
\ba
\bp &&-1/v_1\\&v_1/v_{14}&v_4/v_{14}\\v_1&v_2&v_3&v_4\\&&v_{13}/v_1&v_{14}/v_1\ep = \bp &&1\\&1\\-1\\&&&1\ep \bp -v_1 & -v_2 & -v_3 & -v_4\\ & v_1/v_{14} & v_4/v_{14}\\ && -1/v_1\\ && v_{13}/v_1 & v_{14}/v_1\ep
\ea
has the given invariants.

\item $w = s_\beta s_\alpha s_\beta$: This says $v \sim (0,1,*,*; 1,*,*,*,*,*)$. Then the matrix
\ba
\bp &&&-1/v_2\\&&v_2/v_{12}\\&v_2&v_3&v_4\\-v_{12}/v_2&&v_{23}/v_2&v_{24}/v_2\ep = \bp &&&1\\&&1\\&-1\\-1\ep \bp v_{12}/v_2 && -v_{23}/v_2 & -v_{24}/v_2\\ & -v_2 & -v_3 & -v_4\\ && v_2/v_{12}\\ &&& -1/v_2\ep
\ea
has the given invariants.

\item $w = w_0$: This says $v\sim (1,*,*,*; 1,*,*,*,*,*)$. Then the matrix
\ba
\bp &&-1/v_1\\&&v_2/v_{12}&-v_1/v_{12}\\v_1&v_2&v_3&v_4\\&v_{12}/v_1&v_{13}/v_1&v_{14}/v_1\ep = \bp && 1\\ &&& 1\\ -1\\ &-1\ep \bp -v_1& -v_2 & -v_3 & -v_4\\ & -v_{12}/v_1 & -v_{13}/v_1 & -v_{14}/v_1\\ && -1/v_1\\ && v_2/v_{12} & -v_1/v_{12}\ep
\ea
has the given invariants.
\ee

For $w\in W$, let $\Gamma_w = \Gamma_0 \cap w^{-1} \Gamma_0^\top w$. We also let $U_w = U \cap w^{-1} U^\top w$, and $\ol U_w = U \cap w^{-1} U w$. Then clearly we have $U = U_w \ol U_w = \ol U_w U_w$.
\begin{lem}
$\Gamma_w$ acts freely on $(P_0\cap \Gamma) \bs (\Gamma \cap G_w)$ on the right.
\end{lem}
\begin{proof}
See \cite[Lemma 1.2]{Friedberg1987}.
\end{proof}

Unfolding the coprimality conditions $(v_1,\ldots,v_4) = 1$ and $(v_{12},\ldots,v_{34}) = 1$ of the Plücker coordinates, we obtain for $w\in W$ a complete set of coset representatives $R_w$ for the quotient $(P_0\cap \Gamma) \bs (\Gamma \cap G_w) / \Gamma_w$, which are useful for later computations. Since the Plücker coordinates determine the cosets uniquely, it suffices to express $R_w$ in terms of Plücker coordinates:
\ben
\item $w=\id$: We have
\ba
R_{\id} = \cb{(0,0,1,0;0,0,0,0,0,1)}.
\ea

\item $w=s_\alpha$: We have
\begin{equation}\label{eq:a_R}
R_{s_\alpha} = \cb{(0,0,v_3,v_4;0,0,0,0,0,1)},
\end{equation}
where $v_4\geq 1$ and $v_3\pmod{v_4}$ are such that $(v_3, v_4)=1$.

\item $w=s_\beta$: We have
\begin{equation}\label{eq:b_R}
R_{s_\beta} = \cb{(0,0,1,0;0,0,0,v_{23},0,v_{34})},
\end{equation}
where $v_{23}\geq 1$ and $v_{34}\pmod{v_{23}}$ are such that $(v_{23}, v_{34}) = 1$.

\item $w=s_\alpha s_\beta$. We have
\begin{equation}\label{eq:ab_R}
R_{s_\alpha s_\beta} = \cb{\rb{0,v_2, v_3, v_4; 0,0,0,\frac{v_2}{d},0,-\frac{v_4}{d}}},
\end{equation}
where $v_2 \geq 1$, $v_3$ and $v_4\pmod{v_2}$ are such that $(v_2, v_3, v_4)=1$, and $d=(v_2,v_4)$.

\item $w=s_\beta s_\alpha$. We have
\begin{equation}\label{eq:ba_R}
R_{s_\beta s_\alpha} = \cb{\rb{0,0,-\frac{v_{24}}{d},\frac{v_{14}}{d};0,-v_{24},v_{14},-\frac{v_{24}^2}{v_{14}},v_{24},v_{34}}},
\end{equation}
where $v_{14}\geq 1$, $d=(v_{14}, v_{24})$ and $v_{24}, v_{34}\pmod{v_{14}}$ are such that $v_{14}\mid d^2$ and $(d^2/v_{14}, v_{34}) = 1$. 

\item $w=s_\alpha s_\beta s_\alpha$. We have
\begin{equation}\label{eq:aba_R}
R_{s_\alpha s_\beta s_\alpha} = \cb{\rb{v_1,v_2,v_3,v_4; 0, -\frac{v_1v_2}{d\delta}, \frac{v_1^2}{d\delta}, -\frac{v_2^2}{d\delta}, \frac{v_1v_2}{d\delta}, \frac{v_1v_3+v_2v_4}{d\delta}}}
\end{equation}
where $v_1 \geq 1$ and $v_2, v_3, v_4\pmod{v_1}$ are such that $(v_1,v_2,v_3,v_4)=1$, and $d = (v_1, v_2)$, $\delta = (d, v'_1v_3+v'_2v_4)$. 

\item $w=s_\beta s_\alpha s_\beta$. We have
\begin{equation}\label{eq:bab_R}
R_{s_\beta s_\alpha s_\beta} = \cb{\rb{0,\frac{v_{12}}{d_0},\frac{v_{13}}{d_0},\frac{v_{14}}{d_0}; v_{12}, v_{13}, v_{14}, v_{23}, -v_{13}, -\frac{v_{13}^2+v_{14}v_{23}}{v_{12}}}},
\end{equation}
where $v_{12}\geq 1$ and $v_{13}, v_{14} \pmod{v_{12}}$ are such that $d_1\mid v_{13}^2$, $d_1 = (v_{12}, v_{14})$. Let $v_{12} = d_1 v'_{12}$, $v_{14} = d_1v'_{14}$, $v_{13}^2 = d_1k$, $d_0 = (v_{12}, v_{13}, v_{14})$, $d_1 = d_0 d_q$, $d_0 = d_q t$. Let $a$ be a solution to $av'_{14}\equiv -k\pmod{v'_{12}}$ such that $a$ and $(av'_{14}+k)/v'_{12}$ are both divisible by $t$. Then $v_{23}\pmod{v_{12}}$ is chosen so that $v_{23} = a + rv'_{12}$ with $(r,t)=1$.

\item $w = w_0$. Representatives have the form
\begin{align}\label{eq:w0_R}
R_{w_0} = \cb{\rb{v_1, v_2, v_3, v_4; v_{12}, v_{13}, v_{14}, \frac{v_2v_{13}-v_3v_{12}}{v_1}, -v_{13}, \frac{v_3v_{14}-v_4v_{13}}{v_1}}}, 
\end{align}
where $v_1, v_{12}\geq 1$, and $v_2, v_3, v_4\pmod{v_1}$, $v_{13}, v_{14}\pmod{v_{12}}$ are such that $v_1v_{13}+v_2v_{14}-v_4v_{12}=0$, $v_1\mid v_2v_{13}-v_3v_{12}$, $v_1\mid v_3v_{14}-v_4v_{13}$, and
\ba
 (v_1, v_2, v_3, v_4) = 1, \quad \rb{v_{12}, v_{13}, v_{14}, \frac{v_2v_{13}-v_3v_{12}}{v_1}, \frac{v_3v_{14}-v_4v_{13}}{v_1}} = 1.
\ea
\ee

\subsection{Eisenstein series}

We end the section with alternative expressions for Eisenstein series, using Plücker coordinates. The following theorem says that $E_\alpha(g, \nu, 1)$ and $E_\beta(g,\nu,1)$ can be considered as Epstein zeta functions, and $E_0(g,\nu)$ can be considered as a height zeta function associated with a bi-projective quadratic variety.
\begin{thm}\label{Eisenstein_alt}
We have
\ba
E_0(g,\nu) &= \frac{1}{4} \sum\limits_{v\in V_0(\Z) \text{\rm{ primitive}}} (v_\beta^\top gg^\top v_\beta)^{-\nu_1/2-1} (v_\alpha^\top(g\wedge g)(g\wedge g)^\top v_\alpha)^{\nu_1/2-\nu_2-1/2},\\
E_\alpha (g, \nu, 1) &= \frac{1}{2}\sum\limits_{v_\alpha \in V_\alpha(\Z) \text{\rm{ primitive}}} (v_\alpha^\top (g\wedge g)(g\wedge g)^\top v_\alpha)^{-\nu/2-3/4},\\
E_\beta (g, \nu, 1) &= \frac{1}{2}\sum\limits_{v_\beta \in V_\beta(\Z) \text{\rm{ primitive}}} (v_\beta^\top gg^\top v_\beta)^{-\nu/2-1}.
\ea
where $v_\beta = (v_1, v_2, v_3, v_4)^\top$, $v_\alpha = (v_{12}, v_{13}, v_{14}, v_{23} ,v_{24}, v_{34})^\top$, and $V_0, V_\alpha, V_\beta$ are defined by \eqref{V0_def}, \eqref{Va_def}, \eqref{Vb_def} respectively, and $g \wedge g$ is the exterior square of a matrix $g$, given by
\ba
g\wedge g = \bp g_{ij, kl} \ep_{\substack{1\leq i<j\leq 4\\ 1\leq k<l \leq 4}}, \quad \text{ where } \quad g_{ij,kl} = g_{ik}g_{jl}-g_{il}g_{jk}.
\ea
\end{thm}
\begin{proof}
We first prove the statement for $E_0(g,\nu)$. We start from the expression 
\ba
E_0(g,\nu) = \sum\limits_{\gamma \in (P_0\cap \Gamma) \bs \Gamma} I_0(\gamma g,\nu) = \frac 14 \sum\limits_{v\in V_0(\Z) \text{\rm{ primitive}}} I_0(\gamma g,\nu),
\ea
where $I_0(g,\nu) = y_1^{\nu_1+2} y_2^{2\nu_2-\nu_1+1}$. Suppose $\gamma\in\Gamma$ has Plücker coordinates 
\ba
v = (v_1, \cdots, v_4; v_{12}, \cdots, v_{34}).
\ea
Then it suffices to prove that
\ba
I_0 (\gamma g, \nu) = (v_\alpha^\top (g\wedge g)(g\wedge g)^\top v_\alpha)^{\nu_1/2-\nu_2-1/2} (v_\beta^\top g g^\top v_\beta)^{\nu_2-\nu_1-1/2}.
\ea
Let $\gamma g = nak$ be the Iwasawa decomposition of $\gamma g$ with $n \in U$, $a \in T^+$, and $k\in K$. If we write
\ba
a = \diag(a_1, a_2, a_1^{-1}, a_2^{-1}) \in T^+,
\ea
then $I_0(\gamma g,\nu) = a_1^{\nu_1+2} a_2^{2\nu_2-\nu_1+1}$. So it suffices to find expressions for $a_1$ and $a_2$ in terms of the Plücker coordinates of $\gamma$. Suppose $\gamma g$ has the form
\ba
\gamma g = \bp *&*&*&*\\ *&*&*&*\\b_{31}&b_{32}&b_{33}&b_{34}\\b_{41}&b_{42}&b_{43}&b_{44}\ep = nak. 
\ea
Then $\gamma g (\gamma g)^\top = nak(nak)^\top = na^2 n^\top$. Since $n\in U$ has the form
\ba
n = \bp 1&u&*&*\\&1&*&*\\&&1\\&&-u&1\ep \in U,
\ea
we compute
\ba
na^2n^\top = \bp *&*&*&*\\ *&*&*&*\\ *&*& a_1^{-2} & -ua_1^{-2}\\ *&*&-ua_1^{-2}&u^2a_1^{-2}+a_2^{-2}\ep.
\ea
Evaluating $\gamma g (\gamma g)^\top = \gamma g g^\top \gamma^\top$ yields
\ba
a_1^{-2} &= b_{31}^2+b_{32}^2+b_{33}^2+b_{34}^2,\\
-ua_1^{-2} &= b_{31}b_{41}+b_{32}b_{42}+b_{33}b_{43}+b_{34}b_{44},\\
u^2a_1^{-2}+a_2^{-2} &= b_{41}^2+b_{42}^2+b_{43}^2+b_{44}^2,
\ea 
from which we get
\ba
a_2^{-2} = b_{41}^2+b_{42}^2+b_{43}^2+b_{44}^2 - \frac{(b_{31}b_{41}+b_{32}b_{42}+b_{33}b_{43}+b_{34}b_{44})^2}{b_{31}^2+b_{32}^2+b_{33}^2+b_{34}^2}.
\ea
In particular, we have 
\ba
a_1^{-2} a_2^{-2} &= (b_{31}^2+b_{32}^2+b_{33}^2+b_{34}^2)(b_{41}^2+b_{42}^2+b_{43}^2+b_{44}^2) - (b_{31}b_{41}+b_{32}b_{42}+b_{33}b_{43}+b_{34}b_{44})^2\\
&= \sum\limits_{1\leq i < j\leq 4} (b_{3i}b_{4j}-b_{3j}b_{4i})^2.
\ea
Meanwhile, expanding $\gamma g$, we see that
\ba
\bp b_{31}&b_{32}&b_{33}&b_{34}\ep = v_\beta^\top g.
\ea
Let $g\wedge g$ be the exterior square of $g$. Then we have
\ba
\bp b_{3i}b_{4j}-b_{3j}b_{4i} \ep_{1\leq i<j\leq 4} = v_\alpha^\top (g\wedge g), 
\ea
where we consider $\bp b_{3i}b_{4j}-b_{3j}b_{4i} \ep_{1\leq i<j\leq 4}$ as a row vector. So we can write
\begin{align}\label{eq:a1v_formula}
a_1^{-2} &= v_\beta^\top g g^\top v_\beta,\\
a_1^{-2}a_2^{-2} &= v_\alpha^\top (g\wedge g)(g\wedge g)^\top v_\alpha. \label{eq:a12v_formula}
\end{align}
Hence
\ba
I_0 (\gamma g, \nu) = a_1^{\nu_1+2} a_2^{2\nu_2-\nu_1+1} = \big(v_\alpha^\top (g\wedge g)(g\wedge g)^\top v_\alpha\big)^{\nu_1/2-\nu_2-1/2} (v_\beta^\top g g^\top v_\beta)^{\nu_2-\nu_1-1/2},
\ea
proving the statement for $E_0(g,\nu)$. 

For $E_\alpha(g,\nu,1)$ and $E_\beta(g,\nu,1)$, we have
\ba
E_\alpha(g,\nu,1) &= \frac 12 \sum\limits_{v_\alpha \in V_\alpha(\Z) \text{\rm{ primitive}}} I_\alpha (\gamma g, \nu), & \quad E_\beta(g,\nu,1) &= \sum\limits_{v_\beta \in V_\beta(\Z) \text{\rm{ primitive}}} I_\beta (\gamma g, \nu),
\ea
where $I_\alpha (g, \nu) = (y_1y_2)^{\nu+3/2}$, and $I_\beta (g,\nu) = y_1^{\nu+2}$. Then the statements follow from expressions \eqref{eq:a1v_formula} and \eqref{eq:a12v_formula}, using the same argument.
\end{proof}

\section{Constant terms} \label{section:const_terms}
Let $E_P(g,\nu,f)$ be an Eisenstein series for a standard parabolic $P$. Let $P' = N'M'$ be another standard parabolic subgroup. The \emph{constant term} of $E_P(g,\nu,f)$ along the parabolic $P'$ is defined as
\ba
C_P^{P'}(g,\nu,f) := \int_{N'(\Z)\bs N'(\R)} E_P(\eta g,\nu,f) d\eta,
\ea
where $N'(\Z) = \Gamma \cap N'(\R)$. When $P'=P$, the superscript $P'$ is omitted from the notation.

To compute the constant terms, we also make use of intertwining operators, defined in adelic settings. We follow the setup in \cite{MW1995}. Let $\A$ be the ring of adeles of $\Q$. Let $\pi$ be an irreducible automorphic representation of $M$, and $\phi_\pi$ be an element in $A(N(\A)M(\Q)\bs G(\A))_\pi$, the $\pi$-isotypic part of the space of automorphic forms on $N(\A)M(\Q)\bs G(\A)$ (see \cite[I.2.17]{MW1995}). The Eisenstein series associated to $\phi_\pi$ is then defined to be
\ba
E(\phi_\pi, \pi)(g) := \sum\limits_{\gamma \in P(\Q)\bs G(\Q)} \phi_\pi(\gamma g)
\ea
as a function on $G(\Q)\bs G(\A)$, whenever the series converges.

Now let $w \in G(\Q)$ be such that $wMw^{-1} = M'$. For $g\in G(\A)$, we set
\ba
\mc M(w,\pi)\phi_\pi (g) := \int_{(N'(\Q)\cap w N(\Q) w^{-1})\bs N'(\A)} \phi_\pi(w^{-1}\eta g) d\eta
\ea
whenever the integral is convergent. This defines an intertwining operator 
\ba
\mc M(w,\pi): A(N(\A)M(\Q)\bs G(\A))_\pi \to A(N'(\A)M'(\Q)\bs G(\A))_{w\pi}.
\ea 

Now we are able to state the functional equation of Langlands.
\begin{thm}[Langlands \cite{Langlands1976}] In the setting above,
\ba
\mc M(w', w\pi) \circ \mc M(w,\pi) = \mc M(w'w, \pi).
\ea
\end{thm}

We give a correspondence between adelic and classical definitions of Eisenstein series, in the case $G=\Sp(4)$ and $k = \Q$. We have the strong approximation $g = \delta g_\infty k_0$ for all $g\in G(\A)$, with $\delta \in G(\Q)$, $g_\infty \in G(\R)$, and $k_0\in K$, the maximal compact subgroup of $G(\A)$. 

Let $P_0$ be the minimal parabolic subgroup of $\Sp(4)$ with Levi component $M_0$. For $\nu\in\C^2$, let $\pi_\nu$ be a character on $M_0(\A)$ defined by 
\ba
\pi_\nu(\diag(y_1, y_2, y_1^{-1}, y_2^{-1})) := |y_1|^{\nu_1+2} |y_2|^{2\nu_2-\nu_1+1}.
\ea
Based on the Iwasawa decomposition \eqref{eq:g_iwasawa}, we define 
\ba
\phi_\nu(g) := |y_1|^{\nu_1+2} |y_2|^{2\nu_2-\nu_1+1}.
\ea
Then $\phi_\nu$ is right $K$-invariant, and lies in $A(N_0(\A)M_0(\Q)\bs G(\A))_{\pi_\nu}$. It is then easy to check the following.

\begin{prp}
In the setup above,
\ba
E(\phi_\nu, \pi_\nu)(g) = E_0(g_\infty, \nu).
\ea
\end{prp}

\subsection{Constant term along \texorpdfstring{$P_0$}{P0}}\label{minmin}

We consider the minimal parabolic Eisenstein series
\ba
E_0 (g, \nu) = \sum\limits_{\gamma\in (P_0 \cap \Gamma) \bs \Gamma} I_0(\gamma g,\nu).
\ea
with $I_0(g,\nu) = y_1^{\nu_1+2} y_2^{2\nu_2-\nu_1+1}$. By definition, the constant term of $E_0(g,\nu)$ along $P_0$ is
\ba
C_0(g,\nu) := \int_{N_0(\Z)\bs N_0(\R)} \sum\limits_{\gamma\in (P_0\cap\Gamma) \bs \Gamma}I_0(\gamma\eta g, \lambda) d\eta.
\ea
It is clear from the definition of the integral that the constant term is invariant under left action by $N_0(\R)$. So we may assume that $g$ is a diagonal matrix $\diag(y_1, y_2, y_1^{-1}, y_2^{-1})$. Write 
\ba
\eta = \bp 1 & n_1 & n_2 & n_3\\  & 1 & n_4 & n_5\\ &&1\\ &&-n_1&1\ep \in N_0(\R),
\ea
with the relation $n_3 = n_4+n_1n_5$. Then the integration becomes
\ba
\int_0^1 \int_0^1 \int_0^1 \int_0^1 \sum\limits_{\gamma\in (P_0\cap\Gamma) \bs \Gamma}I_0(\gamma\eta g, \lambda) dn_1dn_2dn_4dn_5.
\ea

We break down the summation over $(P_0\cap\Gamma) \bs \Gamma$ via the Bruhat decomposition:
\ba
E_0(g,\nu) = \sum\limits_{w\in W} E_{0,w}(g,\nu),
\ea
where
\ba
E_{0, w}(g,\nu) = \sum\limits_{\gamma \in (P_0\cap\Gamma) \bs (\Gamma \cap P_0 w P_0)} I_0(\gamma g, \nu),
\ea
and compute the constant term integrals
\ba
C_{0,w}(g,\nu) = \int_{N_0(\Z)\bs N_0(\R)} E_{0, w}(g, \nu).
\ea

Again, let $\phi_\nu(g) = |y_1|^{\nu_1+2} |y_2|^{2\nu_2-\nu_1+1}$. It is straightforward to verify
\begin{prp}
For $g = (g_\infty, 1, 1, \cdots) \in G(\A)$, we have $\mc M(w,\nu) \phi_\nu (g) = C_{0, w^{-1}}(g_\infty, \nu)$.
\end{prp}

This functional equation reduces the calculation of the constant terms to the cases of $w = \id, s_\alpha, s_\beta$.

For $\gamma \in (P_0\cap\Gamma) \bs (\Gamma \cap P_0 w P_0)$, let
\ba
\gamma \eta g \equiv \bp 1 & n'_1 & n'_2 & n'_3\\&1&n'_4&n'_5\\&&1\\&&-n'_1&1\ep \bp y'_1\\&y'_2\\&&{y'_1}^{-1}\\&&&{y'_2}^{-1}\ep \pmod{K}.
\ea
Using the explicit Bruhat decomposition in \Cref{section:Bruhat}, we may express $E_{0,w}(g,\nu)$ using Plücker coordinates:
\ben
\item For $w=\id$, we have
\ba
E_{0,\id}(g,\nu) = {y'_1}^{\nu_1+2} {y'_2}^{2\nu_2-\nu_1+1},
\ea
where $y'_1=y_1$ and $y'_2=y_2$.
\item For $w=s_\alpha$, we have
\ba
E_{0,s_\alpha}(g,\nu) = \sum_{v_4\ge 1} \sum_{(v_3,v_4)=1} {y'_1}^{\nu_1+2} {y'_2}^{2\nu_2-\nu_1+1},
\ea
where
\ba
y'_1 = \frac{y_1y_2}{v_4\sqrt{s_1^2y_2^2+y_1^2}} \quad \text{ and } \quad y'_2 = v_4 \sqrt{s_1^2y_2^2+y_1^2},
\ea
with $s_1 = n_1- v_3/v_4$.
\item For $w=s_\beta$, we have
\ba
E_{0,s_\beta}(g,\nu) = \sum_{v_{23}\ge 1} \sum_{(v_{23},v_{34})=1} {y'_1}^{\nu_1+2} {y'_2}^{2\nu_2-\nu_1+1},
\ea
where
\ba
y'_1 = y_1 \quad \text{ and } \quad y'_2 = \frac{y_2}{v_{23}\sqrt{y_2^4+s_5^2}},
\ea
with $s_5 = n_5 - v_{34}/v_{23}$.
\ee

It is clear that $C_{0,\id} (g,\nu) = y_1^{\nu_1+2} y_2^{2\nu_2-\nu_1+1}$. Now we compute
\begin{equation*}
C_{0,s_\alpha} (g,\nu) = \int_0^1 \sum\limits_{v_4\geq 1} \sum\limits_{(v_3, v_4)=1} y_1^{\nu_1+2} y_2^{\nu_1+2} v_4^{2\nu_2-2\nu_1-1} (s_1^2y_2^2+y_1^2)^{\nu_2-\nu_1-1/2} dn_1.
\end{equation*}
Summing over $v_3$ gives an integral over $\R$:
\begin{equation*}
C_{0,s_\alpha} (g,\nu) = y_1^{\nu_1+2} y_2^{\nu_1+2} \sum\limits_{v_4\geq 1} \varphi(v_4) v_4^{2\nu_2-2\nu_1-1} \int_\R (s_1^2y_2^2+y_1^2)^{\nu_2-\nu_1-1/2} dn_1,
\end{equation*}
where $\varphi$ denotes Euler's totient function. Via \cite[Formula 3.251.2]{GR2007}, the integral evaluates to
\begin{equation*}
C_{0,s_\alpha} (g,\nu) = y_1^{2\nu_2-\nu_1+2} y_2^{\nu_1+1} \sum\limits_{v_4\geq 1} \varphi(v_4) v_4^{2\nu_2-2\nu_1-1} B\rb{\frac{1}{2}, \nu_1-\nu_2},
\end{equation*}
where
\begin{equation}\label{eq:beta}
B(x,y) := \int_0^1 t^{x-1} (1-t)^{y-1} dt = \frac{\Gamma(x)\Gamma(y)}{\Gamma(x+y)}
\end{equation}
denotes the beta function. Hence we compute
\begin{align*}
C_{0,s_\alpha} (g,\nu) &= y_1^{2\nu_2-\nu_1+2} y_2^{\nu_1+1} \frac{\zeta(2\nu_1-2\nu_2)}{\zeta(2\nu_1-2\nu_2+1)} B\rb{\frac{1}{2}, \nu_1-\nu_2}\\
&= y_1^{2\nu_2-\nu_1+2} y_2^{\nu_1+1} \frac{\Lambda(2\nu_1-2\nu_2)}{\Lambda(2\nu_1-2\nu_2+1)},
\end{align*}
where $\Lambda(s) = \pi^{-s/2} \Gamma(s/2) \zeta(s)$ is the completed zeta function.

Analogously,
\ba
C_{0,s_\beta} (g,\nu) = y_1^{\nu_1+2} y_2^{\nu_1-2\nu_2+1} \frac{\Lambda(2\nu_2-\nu_1)}{\Lambda(2\nu_2-\nu_1+1)}.
\ea

In term of interwtining operators, we have
\ba
\mc M(s_\alpha, \nu) I_0(g, \nu) &= \frac{\Lambda(2\nu_1-2\nu_2)}{\Lambda(2\nu_1-2\nu_2+1)} I_0(g, (2\nu_2-\nu_1, \nu_2)),\\
\mc M(s_\beta, \nu) I_0(g,\nu) &= \frac{\Lambda(2\nu_2-\nu_1)}{\Lambda(2\nu_2-\nu_1+1)} I_0(g, (\nu_1, \nu_1-\nu_2)). 
\ea
Through the functional equation, we obtain the constant term for other Weyl elements. This completes the computation of the constant terms. 

\begin{prp}\label{constterm:minmin}
The constant term for the minimal parabolic Eisenstein series along the minimal parabolic subgroup $P_0$ is given by
\ba
C_0 (g,\nu) = \sum\limits_{w \in W} C_{0,w}(g,\nu)
\ea
where
\ba
C_{0, \id} (g,\nu) &=  y_1^{\nu_1+2} y_2^{2\nu_2-\nu_1+1},\\
C_{0, s_\alpha} (g,\nu)  &= \frac{\Lambda(2\nu_1-2\nu_2)}{\Lambda(2\nu_1-2\nu_2+1)} y_1^{2\nu_2-\nu_1+2} y_2^{\nu_1+1},\\
C_{0, s_\beta}  (g,\nu) &= \frac{\Lambda(2\nu_2-\nu_1)}{\Lambda(2\nu_2-\nu_1+1)} y_1^{\nu_1+2} y_2^{\nu_1-2\nu_2+1},\\
C_{0, s_\alpha s_\beta}  (g,\nu) &= \frac{\Lambda(\nu_1)}{\Lambda(\nu_1+1)} \frac{\Lambda(2\nu_1-2\nu_2)}{\Lambda(2\nu_1-2\nu_2+1)} y_1^{2\nu_2-\nu_1+2} y_2^{-\nu_1+1},\\
C_{0, s_\beta s_\alpha}  (g,\nu) &= \frac{\Lambda(2\nu_2)}{\Lambda(2\nu_2+1)}\frac{\Lambda(2\nu_2-\nu_1)}{\Lambda(2\nu_2-\nu_1+1)} y_1^{\nu_1-2\nu_2+2} y_2^{\nu_1+1},\\
C_{0, s_\alpha s_\beta s_\alpha}  (g,\nu) &= \frac{\Lambda(2\nu_2)}{\Lambda(2\nu_2+1)}\frac{\Lambda(\nu_1)}{\Lambda(\nu_1+1)} \frac{\Lambda(2\nu_1-2\nu_2)}{\Lambda(2\nu_1-2\nu_2+1)} y_1^{-\nu_1+2} y_2^{2\nu_2-\nu_1+1},\\
C_{0, s_\beta s_\alpha s_\beta}  (g,\nu) &= \frac{\Lambda(\nu_1)}{\Lambda(\nu_1+1)}\frac{\Lambda(2\nu_2)}{\Lambda(2\nu_2+1)}\frac{\Lambda(2\nu_2-\nu_1)}{\Lambda(2\nu_2-\nu_1+1)} y_1^{\nu_1-2\nu_2+2} y_2^{-\nu_1+1},\\
C_{0, w_0}  (g,\nu) &= \frac{\Lambda(2\nu_2-\nu_1)}{\Lambda(2\nu_2-\nu_1+1)}\frac{\Lambda(2\nu_2)}{\Lambda(2\nu_2+1)}\frac{\Lambda(\nu_1)}{\Lambda(\nu_1+1)} \frac{\Lambda(2\nu_1-2\nu_2)}{\Lambda(2\nu_1-2\nu_2+1)} y_1^{-\nu_1+2} y_2^{\nu_1-2\nu_2+1}.
\ea
\end{prp}
\begin{rmk}
Note that the constant term $C_0(g,\nu)$ is precisely the constant Fourier coefficient $E_{0,\chi_{0,0}}(g,\nu)$. This gives an expression for $E_{0,\chi_{0,0}}(g,\nu)$ nicer than the one given in \Cref{Fourier:min}. Nevertheless, we keep the former expression to maintain the structural consistency of the expressions there.
\end{rmk}

\subsection{Constant term along \texorpdfstring{$P_\alpha$}{Pa}}\label{minSiegel}

We have to express the constant terms by using intertwining operators. A detailed description is given in \cite[II.1.7]{MW1995}. Let $W = W(T, G)$ be the Weyl group, and $W_M = W(T,M)$ the Weyl group corresponding to $M$. We define 
\ba
W(M,M') := \cbm{w\in W}{w^{-1}(\lambda) > 0 \text{ for any positive root $\lambda$ of $M'$ over $A_0$, and $wMw^{-1}\sbe M'$}}.
\ea

In general, if $E(\phi_\pi, \pi)$ is an Eisenstein series along a parabolic $P=MN$, its constant term along the parabolic $P'=M'N'$ is given by
\ba
\int_{P'} E(\phi_\pi, \pi)(g) = \int_{N'(k)\bs N'(\A)} E(\phi_\pi, \pi)(\eta g)d\eta = \int_{N'(k)\bs N'(\A)} \sum\limits_{\gamma\in P(k)\bs G(k)} \phi_\pi(\gamma \eta g) d\eta.
\ea
By \cite[II.1.7]{MW1995}, it can also be expressed via intertwining operators:
\ba
\int_{P'} E(\phi_\pi, \pi)(g) = \sum\limits_{w\in W(M,M')} \sum\limits_{m' \in (M'(k) \cap w P(k) w^{-1}) \bs M'(k)} \mc M(w,\pi)\phi(m'g).
\ea
We compute $W(M_0, M_\alpha) = \{\id, s_\beta, s_\beta s_\alpha, s_\beta s_\alpha s_\beta\}$. By \cite[II.1.7]{MW1995} again, we have
\ba
\int_{N_\alpha(\Z)\bs N_\alpha(\R)} E(\phi_\nu, \nu)(g) = \sum\limits_{w \in W(M_0, M_\alpha)} \sum\limits_{m \in (M_\alpha(\Q) \cap w P(\Q) w^{-1} )\bs M_\alpha(\Q)} \mc M(w,\nu) \phi_\nu(mg).
\ea
Meanwhile
\ba
\mc M(w,\nu) \phi_\nu(mg) &= \int_{(U_\alpha(\Q) \cap w U(\Q) w^{-1})\bs U_\alpha(\A)} \phi_\nu(w^{-1}umg) du\\
&= \int_{U_\alpha(\Q)\bs U_\alpha(\A)} \sum\limits_{u'\in (U_\alpha(\Q)\cap w U(\Q) w^{-1})\bs U_\alpha(\Q)} \phi_\nu(w^{-1}uu'mg) du.
\ea
This is just the constant term integral with respect to $\gamma = w^{-1}mu$ for a given $w\in W(M_0,M_\alpha)$ and $m\in (M_\alpha(\Q) \cap w P(\Q) w^{-1}) \bs M_\alpha(\Q)$. To compute the constant terms, it suffices to find $\phi_\nu(g)$. Now, a set of coset representatives of $(M_\alpha(\Q) \cap w P(\Q) w^{-1})\bs M_\alpha(\Q)$ (which turns out to be independent of $w$) is given by
\ba
\bp 1\\ &1\\&&1\\&&&1\ep \cup \cbm{ m_{\kappa_1, \kappa_2} := \bp & \kappa_2^{-1}\\ \kappa_2 & -\kappa_1\\&&\kappa_1 & \kappa_2\\&&\kappa_2^{-1}\ep }{ \kappa_2 \in \N, (\kappa_1,\kappa_2)=1}.
\ea
These representatives have equivalence with integral matrices (with unit determinant) under the action of $P(\Q)$, so we only have to consider the Archimedean place.
\begin{rmk}
The parameters $\kappa_1, \kappa_2$ are just the Plücker coordinates $v_3, v_4$ in the Bruhat decomposition with $w = s_\alpha$. 
\end{rmk}

When $m=\id$, the constant term integral is identical to those over the minimal parabolic, as the integral is independent of $n_1$. So we have
\ba
\phi_\nu(g) &= y_1^{\nu_1+2} y_2^{2\nu_2-\nu_1+1},\\
\mc M(s_\beta, \nu) \phi_\nu(g) &= \frac{\Lambda(2\nu_2-\nu_1)}{\Lambda(2\nu_2-\nu_1+1)} y_1^{\nu_1+2} y_2^{\nu_1-2\nu_2+1},\\
\mc M(s_\beta s_\alpha, \nu) \phi_\nu(g) &= \frac{\Lambda(\nu_1)}{\Lambda(\nu_1+1)} \frac{\Lambda(2\nu_1-2\nu_2)}{\Lambda(2\nu_1-2\nu_2+1)} y_1^{2\nu_2-\nu_1+2} y_2^{-\nu_1+1},\\
\mc M(s_\beta s_\alpha s_\beta, \nu) \phi_\nu(g) &= \frac{\Lambda(\nu_1)}{\Lambda(\nu_1+1)}\frac{\Lambda(2\nu_2)}{\Lambda(2\nu_2+1)}\frac{\Lambda(2\nu_2-\nu_1)}{\Lambda(2\nu_2-\nu_1+1)} y_1^{\nu_1-2\nu_2+2} y_2^{-\nu_1+1}.
\ea

Now we compute $\mc M(w,\nu)(mg)$ for $m\in (M_\alpha(\Q) \cap w P(\Q) w^{-1})\bs M_\alpha(\Q)$. Let $g = (g_\infty, 1,\cdots)$. Analogously to the constant term computations over the minimal parabolic with $w=s_\alpha$, we see that  if $f(g) = y_1^{c_1} y_2^{c_2}$, then
\ba
f(m_{\kappa_1, \kappa_2}g) = y_1^{c_2} y_2^{c_1} Q(\kappa_1, \kappa_2)^{c_2/2-c_1/2},
\ea
where $Q(\kappa_1, \kappa_2)$ is the quadratic form defined by
\ba
Q(\kappa_1, \kappa_2):= \kappa_1^2-2n_1\kappa_1\kappa_2+\rb{n_1^2+\frac{y_1^2}{y_2^2}}\kappa_2^2.
\ea
Hence
\begin{align*}
\phi_\nu(m_{\kappa_1,\kappa_2}g) &= y_1^{\nu_1+2} y_2^{2\nu_2-\nu_1+1} Q(\kappa_1,\kappa_2)^{\nu_2-\nu_1-1/2},\\
\mc M(s_\beta, \nu)\phi_\nu(m_{\kappa_1,\kappa_2}g) &= \frac{\Lambda(2\nu_2-\nu_1)}{\Lambda(2\nu_2-\nu_1+1)} y_1^{\nu_1+2} y_2^{\nu_1-2\nu_2+1} Q(\kappa_1,\kappa_2)^{-\nu_2-1/2},\\
\mc M(s_\beta s_\alpha, \nu)\phi_\nu(m_{\kappa_1,\kappa_2}g) &= \frac{\Lambda(\nu_1)}{\Lambda(\nu_1+1)} \frac{\Lambda(2\nu_1-2\nu_2)}{\Lambda(2\nu_1-2\nu_2+1)} y_1^{2\nu_2-\nu_1+2} y_2^{-\nu_1+1} Q(\kappa_1,\kappa_2)^{-\nu_2-1/2},\\
\mc M(s_\beta s_\alpha s_\beta, \nu)\phi_\nu(m_{\kappa_1,\kappa_2}g) &= \frac{\Lambda(\nu_1)}{\Lambda(\nu_1+1)}\frac{\Lambda(2\nu_2)}{\Lambda(2\nu_2+1)}\frac{\Lambda(2\nu_2-\nu_1)}{\Lambda(2\nu_2-\nu_1+1)} y_1^{\nu_1-2\nu_2+2} y_2^{-\nu_1+1}\\
&\hspace{0.5cm} \times Q(\kappa_1,\kappa_2)^{\nu_2-\nu_1-1/2},
\end{align*}

The terms then assemble into a $\GL(2)$ Eisenstein series, whose definition we now recall:
\ba
E(z,s) := \frac{1}{2} \sum\limits_{\delta\in\Gamma_{\infty, 2} \bs \Gamma_2} I(\gamma z, s),
\ea
where $\Gamma_2 = \SL(2,\Z)$, $\Gamma_{\infty, 2} =\{(\begin{smallmatrix} 1&b\\&1\end{smallmatrix}) \in\SL(2,\Z)\}$, and $I(z,s) = \im(z)^{s+1/2}$.

\begin{prp}\label{constterm:minSiegel}
The constant term for the minimal parabolic Eisenstein series over the Siegel parabolic subgroup $P_\alpha$ is given by
\ba
C_0^\alpha(g,\nu) = \sum\limits_{w \in W(M_0,M_\alpha)} C_{0,w}^\alpha (g,\nu),
\ea
where
\ba
C_{0,\id}^\alpha (g,\nu) &= E\rb{-n_1+\frac{y_1}{y_2}i, \nu_1-\nu_2} y_1^{\nu_2+3/2} y_2^{\nu_2+3/2},\\
C_{0,s_\beta}^\alpha (g,\nu) &= \frac{\Lambda(2\nu_2-\nu_1)}{\Lambda(2\nu_2-\nu_1+1)} E\rb{-n_1+\frac{y_1}{y_2}i, \nu_2} y_1^{\nu_1-\nu_2+3/2} y_2^{\nu_1-\nu_2+3/2},\\
C_{0,s_\beta s_\alpha}^\alpha (g,\nu) &= \frac{\Lambda(\nu_1)}{\Lambda(\nu_1+1)} \frac{\Lambda(2\nu_1-2\nu_2)}{\Lambda(2\nu_1-2\nu_2+1)} E\rb{-n_1+\frac{y_1}{y_2}i, \nu_2} y_1^{\nu_2-\nu_1+3/2} y_2^{\nu_2-\nu_1+3/2},\\
C_{0,s_\beta s_\alpha s_\beta}^\alpha (g,\nu) &= \frac{\Lambda(\nu_1)}{\Lambda(\nu_1+1)}\frac{\Lambda(2\nu_2)}{\Lambda(2\nu_2+1)}\frac{\Lambda(2\nu_2-\nu_1)}{\Lambda(2\nu_2-\nu_1+1)} E\rb{-n_1+\frac{y_1}{y_2}i, \nu_1-\nu_2} y_1^{-\nu_2+3/2} y_2^{-\nu_2+3/2}.
\ea
\end{prp}

\subsection{Constant term along \texorpdfstring{$P_\beta$}{Pb}}\label{minNS}

Similarly, we compute $W(M_0, M_\beta) = \{\id, s_\alpha, s_\alpha s_\beta, s_\alpha s_\beta s_\alpha\}$, and obtain the following proposition.

\begin{prp}\label{constterm:minNS}
The constant term for the minimal parabolic Eisenstein series over the Jacobi parabolic subgroup $P_\beta$ is given by
\ba
C_0^\beta (g,\nu) = \sum\limits_{w\in W(M_0,M_\beta)} C_{0,w}^\beta (g,\nu),
\ea
where
\ba
C_{0,\id}^\beta (g,\nu) &= E\rb{-n_5+y_2^2i, \nu_2-\frac{\nu_1}{2}} y_1^{\nu_1+2},\\
C_{0,s_\alpha}^\beta (g,\nu) &= \frac{\Lambda(2\nu_1-2\nu_2)}{\Lambda(2\nu_1-2\nu_2+1)} E\rb{-n_5+y_2^2i, \frac{\nu_1}{2}} y_1^{2\nu_2-\nu_1+2},\\
C_{0,s_\alpha s_\beta}^\beta (g,\nu) &= \frac{\Lambda(2\nu_2)}{\Lambda(2\nu_2+1)}\frac{\Lambda(2\nu_2-\nu_1)}{\Lambda(2\nu_2-\nu_1+1)} E\rb{-n_5+y_2^2i, \frac{\nu_1}{2}} y_1^{\nu_1-2\nu_2+2},\\
C_{0,s_\alpha s_\beta s_\alpha}^\beta (g,\nu) &= \frac{\Lambda(2\nu_2)}{\Lambda(2\nu_2+1)}\frac{\Lambda(\nu_1)}{\Lambda(\nu_1+1)} \frac{\Lambda(2\nu_1-2\nu_2)}{\Lambda(2\nu_1-2\nu_2+1)} E\rb{-n_5+y_2^2i, \nu_2-\frac{\nu_1}{2}} y_1^{-\nu_1+2}.
\ea
\end{prp}

\section{$\Sp(4)$ Ramanujan sums} \label{section:Sp4_Ramanujan}

 In the computation of the Fourier coefficients of Eisenstein series, we will come across a sum of the following form:
 \begin{align}\label{eq:Sp4rs_dirichlet} 
\mc R_{\nu_1,\nu_2}(n_1,n_2) := \sum\limits_{v_1, v_{12}\geq 1} v_1^{-\nu_1} v_{12}^{-\nu_2} \sum\limits_{\substack{v_2, v_3, v_4\ppmod{v_1}\\ v_{13}, v_{14}\ppmod{v_{12}}\\ v_1v_{13} + v_2v_{14} - v_4v_{12} \equiv 0\ppmod{v_1v_{12}}\\ (v_1, v_2, v_3, v_4) = 1\\ (v_{12}, v_{13}, v_{14}, v_{23}, v_{34}) = 1}} \e\rb{\frac{n_1 v_2}{v_1} + \frac{n_2v_{14}}{v_{12}}}
 \end{align}
This can be considered as a generalisation of Ramanujan sums, since in the degenerate cases $n_1=0$ or $n_2=0$, the sum reduces to a classical Ramanujan sum, with some extra factors. To state the main result of this section, we introduce the symplectic Schur functions for $\Sp(4,\C)$:
\ba
\Sp_{\lambda_1, \lambda_2} (x_1, x_2) := \frac{\bv x_1^{\lambda_1+2} - x_1^{-(\lambda_1+2)} & x_2^{\lambda_1+2} - x_2^{-(\lambda_1+2)}\\
x_1^{\lambda_2+1}-x_1^{-(\lambda_2+1)} & x_2^{\lambda_2+1}-x_2^{-(\lambda_2+1)}\ev}{\bv x_1^2-x_1^{-2} & x_2^2-x_2^{-2}\\ x_1-x_1^{-1} & x_2-x_2^{-1}\ev} \qquad (\lambda_1\geq \lambda_2\geq 0).
\ea
\begin{rmk}
The terms in $\Sp_{e_1+e_2, e_2}(x_1, x_2)$ correspond to the dimensions of weight spaces of the irreducible representation $V(e_1\omega_1+e_2\omega_2)$ of $\fsp_4(\C)$, which is a special instance of the Weyl character formula (see \cite[Ch. 24]{FH2004}).
\end{rmk}

We also define a multiplicative function $\sigma_{\nu_1, \nu_2}(n_1,n_2)$ by setting for $p$ prime
\begin{equation}\label{eq:sd} 
\sigma_{\nu_1, \nu_2} (p^{e_1}, p^{e_2}) := p^{(e_1+e_2)\nu_1+e_1\nu_2} \Sp_{e_1+e_2, e_1} (p^{\nu_1}, p^{\nu_2}).
\end{equation}

\begin{prp}\label{prp:Sp4rs} 
The sum $\mc R_{\nu_1,\nu_2}(n_1,n_2)$ evaluates as follows.

For $n_1, n_2\ne 0$ we have
\ba
\mc R_{\nu_1,\nu_2}(n_1,n_2) &= \frac{\sigma_{3/2-\nu_1/2-\nu_2, 1/2-\nu_1/2}(n_1, n_2)}{\zeta(\nu_1)\zeta(\nu_2)\zeta(\nu_1+\nu_2-1)\zeta(\nu_1+2\nu_2-2)}.
\ea
For $n_1\ne 0$, $n_2 = 0$ we have
\ba
\mc R_{\nu_1,\nu_2}(n_1,0) &= \frac{\sigma_{1-\nu_1}(n_1)}{\zeta(\nu_1)}\frac{\zeta(\nu_2-1)}{\zeta(\nu_2)}\frac{\zeta(\nu_1+\nu_2-2)}{\zeta(\nu_1+\nu_2-1)}\frac{\zeta(\nu_1+2\nu_2-3)}{\zeta(\nu_1+2\nu_2-2)}.
\ea
For $n_1=0$, $n_2\ne 0$ we have
\ba
\mc R_{\nu_1,\nu_2}(0,n_2) &= \frac{\sigma_{1-\nu_2}(n_2)}{\zeta(\nu_2)}\frac{\zeta(\nu_1-1)}{\zeta(\nu_1)}\frac{\zeta(\nu_1+\nu_2-2)}{\zeta(\nu_1+\nu_2-1)}\frac{\zeta(\nu_1+2\nu_2-3)}{\zeta(\nu_1+2\nu_2-2)}.
\ea
For $n_1=n_2=0$ we have
\ba
\mc R_{\nu_1,\nu_2}(0,0) &= \frac{\zeta(\nu_1-1)}{\zeta(\nu_1)}\frac{\zeta(\nu_2-1)}{\zeta(\nu_2)}\frac{\zeta(\nu_1+\nu_2-2)}{\zeta(\nu_1+\nu_2-1)}\frac{\zeta(\nu_1+2\nu_2-3)}{\zeta(\nu_1+2\nu_2-2)}.
\ea
\end{prp}
\begin{proof}
For fixed $v_1, v_2, v_{12}, v_{14}$, define
\begin{multline*}
L_{v_1, v_{12}} (v_2, v_{14}) := \{ (v_3\ppmod{v_1}, \; v_4\ppmod{v_1},\; v_{13}\ppmod{v_{12}}) \mid\\
v_1v_{13} + v_2v_{14} - v_4v_{12}  \equiv 0 \ppmod{v_1v_{12}}, \; (v_1, v_2, v_3, v_4) = 1, \; (v_{12}, v_{13}, v_{14}, v_{23}, v_{34}) = 1 \},
\end{multline*}
and
\ba
\Lambda_{v_1, v_{12}} (v_2, v_{14}) := \vb{L_{v_1, v_{12}} (v_2, v_{14})}.
\ea
Further, define
\begin{equation}\label{eq:Sp4rsd} 
R_{v_1, v_{12}} (n_1, n_2) := \sum\limits_{\substack{v_2\ppmod{v_1}\\ v_{14}\ppmod{v_{12}}}} \Lambda_{v_1, v_{12}} (v_2, v_{14}) \e\rb{\frac{n_1 v_2}{v_1} + \frac{n_2v_{14}}{v_{12}}}.
\end{equation}
Then we can rewrite \ref{eq:Sp4rs_dirichlet} as
\ba
\mc R_{\nu_1,\nu_2} (n_1,n_2) =  \sum\limits_{v_1, v_{12}\geq 1} v_1^{-\nu_1} v_{12}^{-\nu_2} R_{v_1, v_{12}} (n_1, n_2).
\ea
Now define
\ba
r_{v_1, v_{12}} (n_1, n_2) := \sum\limits_{\substack{u_1\mid v_1\\ u_{12}\mid v_{12}}} R_{u_1, u_{12}} (n_1, n_2).
\ea
We expand
\ba
r_{v_1, v_{12}} (n_1, n_2) &= \sum\limits_{\substack{u_1\mid v_1\\ u_{12}\mid v_{12}}} R_{u_1, u_{12}} (n_1, n_2)\\
&= \sum\limits_{\substack{u_1\mid v_1\\ u_{12}\mid v_{12}}} \sum\limits_{\substack{u_2\ppmod{u_1}\\ u_{14}\ppmod{u_{12}}}} \sum\limits_{\substack{ u_3, u_4\ppmod{u_1}\\ u_{13}\ppmod{u_{12}}\\ u_1u_{13}+u_2u_{14}-u_4u_{12} \equiv 0 \ppmod{u_1u_{12}}\\ (u_1, u_2, u_3, u_4) = 1\\ (v_{12}, v_{13}, v_{14}, v_{23}, v_{34}) = 1}} \e\rb{\frac{n_1u_2}{u_1}+\frac{n_2u_{14}}{u_{12}}}.
\ea
Find $d_1, d_{12}$ such that $v_1 = u_1 d_1$, $v_{12} = u_{12} d_{12}$, and let $v_2 = u_2 d_1$, $v_3 = u_3 d_1$, $v_4 = u_4 d_1$, $v_{13} = u_{13} d_{12}$, $v_{14} = u_{14} d_{12}$. Then the sum becomes
\ba
r_{v_1, v_{12}} (n_1, n_2) &= \sum\limits_{\substack{d_1\mid v_1\\ d_{12}\mid v_{12}}} \sum\limits_{\substack{v_2\ppmod{v_1}\\ v_{14}\ppmod{v_{12}}}} \sum\limits_{\substack{ v_3, v_4\ppmod{v_1}\\ v_{13}\ppmod{v_{12}}\\ v_1v_{13}+v_2v_{14}-v_4v_{12}\equiv 0\ppmod{v_1v_{12}}\\ (v_1, v_2, v_3, v_4) = d_1\\ (v_{12}, v_{13}, v_{14}, v_{23}, v_{34}) = d_{12}}} \e\rb{\frac{n_1 v_2}{v_1} + \frac{n_2v_{14}}{v_{12}}}\\
&= \sum\limits_{\substack{v_2\ppmod{v_1}\\ v_{14}\ppmod{v_{12}}}} \sum\limits_{\substack{ v_3, v_4\ppmod{v_1}\\ v_{13}\ppmod{v_{12}}\\ v_1v_{13}+v_2v_{14}-v_4v_{12}\equiv 0\ppmod{v_1v_{12}}\\ v_{23}, v_{34}\in\Z}} \e\rb{\frac{n_1 v_2}{v_1} + \frac{n_2v_{14}}{v_{12}}},
\ea
so we get rid of the coprimality condition. Note that $v_{23}, v_{34}\in\Z$ is equivalent to
\ba
v_1\mid v_2v_{13}-v_3v_{12}, \quad v_1\mid v_3v_{14}-v_4v_{13}.
\ea
Fixing $v_1, v_2, v_{12}, v_{14}$, we want to find the size of the set
\begin{multline*}
S(v_1, v_{12}, v_2, v_{14}) := \{(v_3\ppmod{v_1},\; v_4\ppmod{v_1},\; v_{13}\ppmod{v_{12}} \mid\\
v_1v_{13} + v_2v_{14} - v_4v_{12} \equiv 0 \ppmod{v_1v_{12}}, \; v_1\mid v_2v_{13}-v_3v_{12}, \; v_1\mid v_3v_{14}-v_4v_{13}\}.
\end{multline*}

This is actually a local problem. Let $v_1 = p^{w_1}, v_2 = p^{w_2}, v_{12} = p^{w_{12}}, v_{14} = p^{w_{14}}$. We may assume $w_2\leq w_1, w_{14}\leq w_{12}$. Note that we need to have $w_2+w_{14}\geq \min\cb{w_1, w_{12}}$ for $S(p^{w_1}, p^{w_{12}}, p^{w_2}, p^{w_{14}})$ to be non-empty. Let $d = \min\cb{w_1, w_{14}}$. Assuming $w_2+w_{14}\geq \min\cb{w_1, w_{12}}$, solving the congruence conditions gives:

\ben
\item For $w_1\leq w_{12}$,
\ber
\item if $2w_1-2w_2>w_{14}$, then
\ba
\vb{S(p^{w_1}, p^{w_{12}}, p^{w_2}, p^{w_{14}})} = \begin{cases} 0 & \text{if } w_{12} > w_2 + w_{14},\\
p^{w_2+w_{14}} & \text{if } w_{12}\leq w_2 + w_{14};\end{cases}
\ea
\item if $2w_1-2w_2\leq w_{14}$, then:
\bea
\item if $2w_1-w_2-w_{12}\geq 0$,
\ber
\item if $d+w_1+w_{12}-2w_2-2w_{14}\geq 1$, then 
\ba
\vb{S(p^{w_1}, p^{w_{12}}, p^{w_2}, p^{w_{14}})} = 2p^{w_2+w_{14}};
\ea
\item if $d+w_1+w_{12}-2w_2-2w_{14} = 0$ or $-1$, then 
\ba
\vb{S(p^{w_1}, p^{w_{12}}, p^{w_2}, p^{w_{14}})} = p^{w_1+w_{12}-w_2-w_{14}+d};
\ea
\item if $d+w_1+w_{12}-2w_2-2w_{14} \leq -2$, then 
\ba
\vb{S(p^{w_1}, p^{w_{12}}, p^{w_2}, p^{w_{14}})} = p^{\lfloor(w_1+w_{12}+d)/2\rfloor};
\ea
\ee
\item if $2w_1-w_2-w_{12}<0$, then
\ber
\item if $d+w_1+w_{12}-2w_2-2w_{14}\geq 1$, then 
\ba
\vb{S(p^{w_1}, p^{w_{12}}, p^{w_2}, p^{w_{14}})} = \begin{cases} 2p^{w_2+w_{14}} & \text{if } w_2+w_{14}\geq w_{12},\\
(p^{w_2+w_{14}}, p^{w_1+d}) & \text{if } w_2+w_{14}<w_{12}; \end{cases}
\ea
\item if $d+w_1+w_{12}-2w_2-2w_{14} = 0$ or $-1$, then 
\ba
\vb{S(p^{w_1}, p^{w_{12}}, p^{w_2}, p^{w_{14}})} = \begin{cases} p^{w_1+w_{12}-w_2-w_{14}+d} & \text{if } w_2+w_{14}\geq w_{12},\\ p^{w_1+d} & \text{if } w_2+w_{14}<w_{12}; \end{cases}
\ea
\item if $d+w_1+w_{12}-2w_2-2w_{14} \leq -2$, then 
\ba
\vb{S(p^{w_1}, p^{w_{12}}, p^{w_2}, p^{w_{14}})} = \begin{cases} (p^{\lfloor(w_1+w_{12}+d)/2\rfloor}, p^{w_1+d}) &\text{if } w_2+w_{14}>w_{12},\\ 
p^{w_1+w_{12}-w_2-w_{14}+d} & \text{if } w_2+w_{14} = w_{12},\\
p^{w_1+d} & \text{if } w_2+w_{14}<w_{12}. \end{cases}
\ea
\ee
\ee
\ee
\item For $w_1\geq w_{12}$,
\ber
\item if $w_{12}\geq w_2$ and $w_{14}\leq 2w_{12}-2w_2$, then
\ba
\vb{S(p^{w_1}, p^{w_{12}}, p^{w_2}, p^{w_{14}})} = p^{w_2+w_{14}};
\ea
\item if $w_{12}\geq w_2$ or $w_{14}\leq 2w_{12}-2w_2$, then
\ba
\vb{S(p^{w_1}, p^{w_{12}}, p^{w_2}, p^{w_{14}})} = p^{w_{12}+\lfloor w_{14}/2\rfloor}.
\ea
\ee
\ee

Now consider the expression
\ba
r_{v_1, v_{12}} (n_1, n_2) = \sum\limits_{\substack{v_2\ppmod{v_1}\\ v_{14}\ppmod{v_{12}}}} \vb{S(v_1,v_{12},v_2,v_{14})} \e\rb{\frac{n_1v_2}{v_1}+\frac{n_2v_{14}}{v_{12}}}. 
\ea
Since $\vb{S(v_1,v_{12},v_2,v_{14})}$ is multiplicative, we deduce that $r_{v_1, v_{12}}(n_1, n_2)$ is multiplicative with respect to $v_1, v_{12}$, in the sense that if $(u_1u_{12},v_1v_{12}) = 1$, then
\ba
r_{u_1v_1, u_{12}v_{12}}(n_1,n_2) = r_{u_1,u_{12}}(n_1,n_2) r_{v_1,v_{12}}(n_1,n_2).
\ea
Indeed, we see that $r_{u_1v_1, u_{12}v_{12}} (n_1, n_2)$ equals
 \ba
&\sum\limits_{\substack{t_2\ppmod{u_1v_1}\\ t_{14}\ppmod{u_{12}v_{12}}}} \vb{S(u_1v_1, u_{12}v_{12}, t_2, t_{14})} \e\rb{\frac{n_1t_2}{u_1v_1} + \frac{n_2t_{14}}{u_{12}v_{12}}}\\
&\hspace{0.5cm}=\sum\limits_{\substack{u_2\ppmod{u_1}\\ u_{14}\ppmod{u_{12}}}} \sum\limits_{\substack{v_2\ppmod{v_1}\\ v_{14}\ppmod{v_{12}}}} \vb{S(u_1v_1, u_{12}v_{12}, u_1v_2+v_1u_2, u_{12}v_{14}+v_{12}u_{14})}\\
&\hspace{6.7cm}\times\e\rb{\frac{n_1u_2}{u_1} + \frac{n_1v_2}{v_1} + \frac{n_2u_{14}}{u_{12}} + \frac{n_2v_{14}}{v_{12}}}\\
&\hspace{0.5cm}=\sum\limits_{\substack{u_2\ppmod{u_1}\\ u_{14}\ppmod{u_{12}}}} \sum\limits_{\substack{v_2\ppmod{v_1}\\ v_{14}\ppmod{v_{12}}}} \vb{S(u_1, u_{12}, v_1u_2, v_{12} u_{14})} \e\rb{\frac{n_1u_2}{u_1}+\frac{n_2u_{14}}{u_{12}}}\\
&\hspace{5.5cm}\times\vb{S(v_1, v_{12}, u_1v_2, u_{12}v_{14})} \e\rb{\frac{n_1v_2}{v_1}+\frac{n_2v_{14}}{v_{12}}}\\
&\hspace{0.5cm}= r_{u_1, u_{12}}(n_1, n_2) r_{v_1, v_{12}} (n_1, n_2)
\ea
as desired. Also, it is clear from definition that if $(m_1m_2, v_1v_{12}) = 1$, then
\ba
r_{v_1,v_{12}} (m_1n_1,m_2n_2) = r_{v_1,v_{12}} (n_1,n_2).
\ea
Thus we have a decomposition
\ba
r_{v_1, v_{12}} (n_1, n_2) = \prod\limits_p r_{p^{\ord_p (v_1)}, p^{\ord_p (v_{12})}} (p^{\ord_p (n_1)}, p^{\ord_p (n_2)}),
\ea
and it suffices to consider the case where $v_1 = p^{w_1}, v_{12} = p^{w_{12}}, n_1 = p^{e_1}, n_2 = p^{e_2}$. Rewrite the expression:
\begin{multline*}
r_{p^{w_1}, p^{w_{12}}} (p^{e_1}, p^{e_2}) = \sum\limits_{w_2 = 0}^{w_1} \sum\limits_{w_{14}=0}^{w_{12}} \vb{S(p^{w_1}, p^{w_{12}}, p^{w_2}, p^{w_{14}})}\\
\times \sum\limits_{\substack{v_2\ppmod{p^{w_1}}\\ \ord_p (v_2) = w_2}} \sum\limits_{\substack{v_{14}\ppmod{p^{w_{12}}}\\ \ord_p(v_{14}) = w_{14}}} \e(v_2 p^{e_1-w_1} + v_{14} p^{e_2-w_{12}}).
\end{multline*}
Without loss of generality, we may assume $e_1\leq w_1, e_2\leq w_{12}$. Noting that
\ba
\sum\limits_{\substack{v\ppmod{p^w}\\ \ord_p (v) = w'}} \e(vp^{e-w}) = 
\begin{cases}
1 & \text{if } w'=w,\\
p^{w-w'-1}(p-1) & \text{if } w> w' \geq w-e,\\
-p^{w-w'-1} & \text{if } w' = w-e-1,\\
0 & \text{if } w'\leq w-e-2,
\end{cases}
\ea
we see that $r_{p^{w_1}, p^{w_{12}}} (p^{e_1}, p^{e_2})$ can be computed explicitly in terms of powers of $p$. Comparing the coefficients then yields
\begin{multline*}
\sum\limits_{w_1,w_{12}\geq 0} r_{p^{w_1}, p^{w_{12}}} (p^{e_1}, p^{e_2}) p^{-w_1\nu_1-w_{12}\nu_2}\\
= \sigma_{3/2-\nu_1/2-\nu_2, 1/2-\nu_1/2}(p^{e_1}, p^{e_2})(1-p^{1-\nu_1-\nu_2})(1-p^{2-\nu_1-2\nu_2}).
\end{multline*}
Combining the $p$-parts gives
\ba
\sum\limits_{v_1,v_{12}\geq 1} r_{v_1, v_2}(n_1, n_2) v_1^{-\nu_1} v_{12}^{-\nu_2} = \frac{\sigma_{3/2-\nu_1/2-\nu_2, 1/2-\nu_1/2}(n_1, n_2)}{\zeta(\nu_1+\nu_2-1)\zeta(\nu_1+2\nu_2-2)} 
\ea
for $n_1, n_2\neq 0$. As
\ba
\sum\limits_{v_1,v_{12}\geq 1} r_{v_1, v_{12}} (n_1, n_2) v_1^{-\nu_1} v_{12}^{-\nu_2} = \zeta(\nu_1)\zeta(\nu_2) \sum\limits_{v_1,v_{12}\geq 1} R_{v_1, v_{12}} (n_1, n_2) v_1^{-\nu_1} v_{12}^{-\nu_2},
\ea
we finally arrive at 
\ba
\sum\limits_{v_1,v_{12}\geq 1} R_{v_1, v_{12}} (n_1, n_2) v_1^{-\nu_1} v_{12}^{-\nu_2} = \frac{\sigma_{3/2-\nu_1/2-\nu_2, 1/2-\nu_1/2}(n_1, n_2)}{\zeta(\nu_1)\zeta(\nu_2)\zeta(\nu_1+\nu_2-1)\zeta(\nu_1+2\nu_2-2)}. 
\ea

Passing to the degenerate cases, we observe that the generalised divisor sum 
\[
\sigma_{3/2-\nu_1/2-\nu_2, 1/2-\nu_1/2}(n_1, n_2)
\]
reduces to classical divisor sums, and we obtain the following formulae. For $n_1\ne 0$, $n_2=0$ we have
\ba
\sum\limits_{v_1,v_{12}\geq 1}  R_{v_1, v_{12}}(n_1, 0) v_1^{-\nu_1} v_{12}^{-\nu_2} &= \frac{\sigma_{1-\nu_1}(n_1)}{\zeta(\nu_1)}\frac{\zeta(\nu_2-1)}{\zeta(\nu_2)}\frac{\zeta(\nu_1+\nu_2-2)}{\zeta(\nu_1+\nu_2-1)}\frac{\zeta(\nu_1+2\nu_2-3)}{\zeta(\nu_1+2\nu_2-2)}.
\ea
For $n_1=0$, $n_2\ne 0$ we have
\ba
\sum\limits_{v_1,v_{12}\geq 1}  R_{v_1, v_{12}}(0, n_2) v_1^{-\nu_1} v_{12}^{-\nu_2} &= \frac{\sigma_{1-\nu_2}(n_2)}{\zeta(\nu_2)}\frac{\zeta(\nu_1-1)}{\zeta(\nu_1)}\frac{\zeta(\nu_1+\nu_2-2)}{\zeta(\nu_1+\nu_2-1)}\frac{\zeta(\nu_1+2\nu_2-3)}{\zeta(\nu_1+2\nu_2-2)}.
\ea
For $n_1=n_2=0$ we have
\ba
\sum\limits_{v_1,v_{12}\geq 1} R_{v_1, v_{12}}(0,0) v_1^{-\nu_1} v_{12}^{-\nu_2} &= \frac{\zeta(\nu_1-1)}{\zeta(\nu_1)}\frac{\zeta(\nu_2-1)}{\zeta(\nu_2)}\frac{\zeta(\nu_1+\nu_2-2)}{\zeta(\nu_1+\nu_2-1)}\frac{\zeta(\nu_1+2\nu_2-3)}{\zeta(\nu_1+2\nu_2-2)}.
\ea
This completes the proof of \Cref{prp:Sp4rs}.
\end{proof}

\section{Fourier coefficients of eisenstein series} \label{section:Fourier_coeff}

\subsection{Invariant differential operators}

Consider the Siegel upper half-space of degree $2$:
\ba
H_2 = \cbm{Z = X+iY\in M_2(\C)}{Y>0}.
\ea
If we write
\ba
Z &= \bp Z_1 & Z_2 \\ Z_2 & Z_3 \ep, & \quad Z_j &= X_j + iY_j, & \quad j&=1,2,3,
\ea
then the generators $\Delta_1, \Delta_2$ of $\Sp(4,\R)$-invariant differential operators on $H_2$ are given in \cite{Niwa1991} by
\ba
\Delta_1 &= \sum\limits_{i,j=1}^3 Y_i Y_j \partial_i \ol\partial_j - D\rb{\partial_1\ol\partial_3 + \ol\partial_1\partial_3 - \frac{1}{2} \partial_2\ol{\partial_2}},\\
\Delta_2 &= D^2\rb{\partial_1\partial_3-\frac{1}{4}\partial_2^2}\rb{\ol\partial_1\ol\partial_3-\frac{1}{4}\ol\partial_2^2} + \frac{i}{4} D\bigg(\sum\limits_{i=1}^3 Y_i\partial_i\bigg)\rb{\ol\partial_1\ol\partial_3-\frac{1}{4}\ol\partial_2^2}\\
 &\hspace{0.5cm}+ \frac{i}{4} D\bigg(\sum\limits_{i=1}^3 Y_i\ol\partial_i\bigg)\rb{\partial_1\partial_3-\frac{1}{4}\partial_2^2} + \frac{1}{16} D\rb{\partial_1\ol\partial_3 + \ol\partial_1\partial_3-\frac{1}{2} \partial_2\ol\partial_2},
\ea
where $D = Y_1Y_3-Y_2^2$, and for $j=1,2,3$,
\ba
\partial_j = \pd{}{Z_j} &= \frac{1}{2}\rb{\pd{}{X_j} - i\pd{}{Y_j}}, \quad & \quad \ol\partial_j = \pd{}{\ol Z_j} &= \frac{1}{2}\rb{\pd{}{X_j} + i\pd{}{Y_j}}.
\ea

Through the isomorphism
\ba
gK \mapsto g\bp i\\&i\ep \quad \text{(symplectic transformation)}
\ea
from $G/K$ to $H_2$, we can consider $\Delta_1, \Delta_2$ as differential operators on $G/K$. It is straightforward to verify that $I_0(g, \nu) = y_1^{\nu_1+2} y_2^{2\nu_2-\nu_1+1}$ is an eigenfunction for $\Delta_1$ and $\Delta_2$, with eigenvalues given by
\ba
\lambda_{\Delta_1} &= \frac{1}{16} (2\nu_1^2-4\nu_1\nu_2+4\nu_2^2-5),\\
\lambda_{\Delta_2} &= \frac{1}{256}(\nu_1^2-2\nu_1\nu_2-2)(\nu_1-2\nu_2-2)(\nu_1+2).
\ea

\subsection{Jacquet's Whittaker functions} \label{section:Whittaker}

It is easily verified that a character $\chi$ on $U(\Z)\bs U(\R)$ has the form
\begin{equation}\label{eq:cd} 
\chi\bp 1 & n_1 & n_2 & n_3\\ & 1& n_4 & n_5\\ && 1\\ &&-n_1 & 1\ep = \e(m_1n_1+m_2n_5)
\end{equation}
for some $m_1,m_2\in\Z$. We shall denote such a character by $\chi_{m_1,m_2}$. 

Now we consider functions $F$ on $G/K$ satisfying the following properties:
\ber
\item $F$ is an eigenfunction for $\Delta_1$ and $\Delta_2$, with the same eigenvalues as $I_0(g,\nu)$;
\item $F(\eta g) = \chi(\eta) F(g)$ for all $\eta\in U(\R)$.
\ee
The space of functions satisfying (i) and (ii) is denoted by $\mc W(\nu,\chi)$. Since $\Delta_1, \Delta_2$ are $\Sp(4,\R)$-invariant differential operators, it follows that for every $w\in W$, $I_0(wg, \nu)$ is also an eigenfunction for $\Delta_1$ and $\Delta_2$ with the same eigenvalues as $I_0(g, \nu)$. For $w\in W$, if the character $\chi$ is trivial on $\ol U_w(\R)$, we define
\begin{align}\label{eq:wd} 
W_w (g, \nu, \chi) := \int_{U_w(\R)} I_0(w \eta g, \nu) \bar\chi(\eta) d\eta \in \mc W(\nu, \chi).
\end{align}
The functions $W_w (g, \nu, \chi)$ are known as \emph{Jacquet's Whittaker functions}; their properties are studied extensively in \cite{Hashizume1982,Jacquet1967,Kostant1978}. If $\chi$ is not trivial on $\ol U_w(\R)$, then we define $W_w (g, \nu, \chi) := 0$. Using the standard Iwasawa decomposition for $\Sp(4)$, we obtain explicit formulae for $W_w(g,\nu,\chi)$ for $g$ of the form \eqref{eq:g_iwasawa}, and $\chi = \chi_{m_1,m_2}$.

\ben
\item $w=\id$: It is easy to see that
\[
W_{\id} (g,\nu,\chi) = y_1^{\nu_1+2} y_2^{2\nu_2-\nu_1+1}
\]
if $m_1=m_2=0$, and $W_{\id} (g,\nu,\chi)=0$ otherwise.

\item $w=s_\alpha$: We compute
\begin{equation}\label{eq:a_W}
W_{s_\alpha}(g,\nu,\chi) = y_1^{\nu_1+2} y_2^{\nu_1+2} \int_\R (n_1^2y_2^2+y_1^2)^{\nu_2-\nu_1-1/2} \e(-m_1n_1) dn_1
\end{equation}
if $m_2=0$, and $W_{s_\alpha}(g,\nu,\chi) = 0$ otherwise.

\item $w=s_\beta$: We compute
\begin{equation}\label{eq:b_W}
W_{s_\beta} (g,\nu,\chi) = y_1^{\nu_1+2} y_2^{2\nu_2-\nu_1+1} \int_\R (y_2^4+n_5^2)^{\nu_1/2-\nu_2-1/2} \e(-m_2n_5) dn_5
\end{equation}
if $m_1=0$, and $W_{s_\beta} (g,\nu,\chi) = 0$ otherwise.

\item $w=s_\alpha s_\beta$: We compute
\begin{equation}\label{eq:ab_W}
\begin{aligned}
W_{s_\alpha s_\beta}(g,\nu,\chi) &= y_1^{\nu_1+2} y_2^{\nu_1+2} \int_\R \int_\R (y_2^4+n_5^2)^{\nu_1/2-\nu_2-1/2}\\ &\hspace{0.5cm}\times \rb{(y_2^4+n_5^2)y_1^2 + y_2^2n_4^2}^{\nu_2-\nu_1-1/2}\e(-m_2n_5) dn_4 dn_5
\end{aligned}
\end{equation}
if $m_1=0$, and $W_{s_\alpha s_\beta}(g,\nu,\chi) = 0$ otherwise.

\item $w=s_\beta s_\alpha$: We compute
\begin{equation}\label{eq:ba_W}
\begin{aligned}
W_{s_\beta s_\alpha}(g,\nu,\chi) &= y_1^{\nu_1+2} y_2^{\nu_1+2} \int_\R \int_\R (n_1^2y_2^2+y_1^2)^{\nu_2-\nu_1-1/2}\\
&\hspace{0.5cm} \times \rb{n_2^2+(n_1^2y_2^2+y_1^2)^2}^{\nu_1/2-\nu_2-1/2}\e(-m_1n_1) dn_1 dn_2
\end{aligned}
\end{equation}
if $m_2=0$, and $W_{s_\beta s_\alpha} (g,\nu,\chi) = 0$ otherwise.

\item $w=s_\alpha s_\beta s_\alpha$: We compute
\begin{equation}\label{eq:aba_W}
\begin{aligned}
W_{s_\alpha s_\beta s_\alpha} (g,\nu,\chi) &= y_1^{\nu_1+2} y_2^{\nu_1+2} \int_\R \int_\R \int_\R \rb{(y_1^2+n_1^2y_2^2)^2+(n_2+n_1n_4)^2}^{\nu_1/2-\nu_2-1/2}\\
&\hspace{0.5cm}\times (y_1^4y_2^2+n_2^2y_2^2+n_1^2y_1^2y_2^4+n_4^2y_1^2)^{\nu_2-\nu_1-1/2} \e(-m_1n_1) dn_1 dn_2 dn_4
\end{aligned}
\end{equation}
if $m_2=0$, and $W_{s_\alpha s_\beta s_\alpha} (g,\nu,\chi) = 0$ otherwise.

\item $w=s_\beta s_\alpha s_\beta$: We compute
\begin{equation}\label{eq:bab_W}
\begin{aligned}
W_{s_\beta s_\alpha s_\beta} (g,\nu,\chi) &= y_1^{\nu_1+2} y_2^{\nu_1+2} \int_\R \int _\R \int_\R (y_1^2 y_2^4+n_5^2 y_1^2+n_4^2 y_2^2)^{\nu_2-\nu_1-1/2}\\
&\hspace{0.5cm}\times\big(y_1^4 y_2^4+n_5^2 y_1^4 +2 n_4^2 y_1^2 y_2^2 +(n_1 n_4- n_2)^2 y_2^4\\
&\hspace{1cm}+(n_2n_5 - n_4^2 - n_1n_4n_5)^2\big)^{\nu_1/2-\nu_2-1/2} \e(-m_2n_5) dn_2 dn_4 dn_5
\end{aligned}
\end{equation}
if $m_1=0$, and $W_{s_\beta s_\alpha s_\beta}(g,\nu,\chi) = 0$ otherwise.

\item $w=w_0$: We compute
\begin{multline}
W_{w_0} \rb{g,\nu,\chi} = y_1^{\nu_1+2}y_2^{\nu_1+2}\int_\R \int_\R \int_\R \int_\R (n_1^2y_1^2y_2^4+y_1^4y_2^2+n_3^2y_1^2+n_2^2y_2^2)^{\nu_2-\nu_1-1/2}\\
\times(n_1^2n_4^2y_2^4+y_1^4y_2^4-2n_1n_5n_4y_1^2y_2^2-2n_1n_2n_4y_2^4 +n_5^2y_1^4+2n_3n_4y_1^2y_2^2+n_2^2y_2^4\\
+n_2^2n_5^2-2n_3n_2n_5n_4+n_3^2n_4^2)^{\nu_1/2-\nu_2-1/2}\e(-m_1n_1-m_2n_5) dn_1 dn_2 dn_4 dn_5.
\end{multline}
\ee

With the exception of the long element $w = w_0$, $W_w$ can be expressed in terms of the classical Whittaker function
\ba
W(y, \nu, \psi) = \int_\R \rb{\frac{y}{y^2+u^2}}^{\nu+\frac 12} \bar\psi(u) du,
\ea
where $\psi = \psi_t(u) = \e(tu)$ for $t\in\R$ is an additive character of $\R$. 

\begin{prp}
We have
\ba
W_{\id} (g,\nu,\chi_{0,0}) &= y_1^{\nu_1+2} y_2^{2\nu_2-\nu_1+1},\\
W_{s_\alpha} (g,\nu,\chi_{m_1,0}) &= y_1^{\nu_2+3/2} y_2^{\nu_1+1} W(y_1, \nu_1-\nu_2, \psi_{m_1/y_2}),\\
W_{s_\beta} (g,\nu,\chi_{0,m_2}) &= y_1^{\nu_1+2} W\rb{y_2^2, \nu_2-\frac{\nu_1}{2}, \psi_{m_2}},\\
W_{s_\alpha s_\beta} (g,\nu,\chi_{0,m_2}) &= y_1^{2\nu_2-\nu_1+2} B\rb{\frac{1}{2}, \nu_1-\nu_2} W\rb{y_2^2, \frac{\nu_1}{2}, \psi_{m_2}},\\
W_{s_\beta s_\alpha} (g,\nu,\chi_{m_1,0}) &= y_1^{\nu_1-\nu_2+3/2} y_2^{\nu_1+1} B\rb{\frac{1}{2}, \nu_2-\frac{\nu_1}{2}} W(y_1, \nu_2, \psi_{m_1/y_2}),\\
W_{s_\alpha s_\beta s_\alpha} (g,\nu,\chi_{m_1,0}) &= y_1^{\nu_2-\nu_1+3/2} y_2^{2\nu_2-\nu_1+1} B\rb{\frac{1}{2}, \frac{\nu_1}{2}} B\rb{\frac{1}{2}, \nu_1-\nu_2} W(y_1, \nu_2, \psi_{m_1/y_2}),\\
W_{s_\beta s_\alpha s_\beta}(g,\nu,\chi_{0,m_2}) &= y_1^{\nu_1-2\nu_2+2} B\rb{\frac{1}{2}, \nu_2-\frac{\nu_1}{2}} B\rb{\frac{1}{2}, \nu_2} W\rb{y_2^2, \frac{\nu_1}{2}, \psi_{m_2}},
\ea
where $B(x,y)$ is the beta function (see \eqref{eq:beta}).
\end{prp}
\begin{proof}
We start with the explicit formulae above and simplify the integrals. 
\ben
\item The statement for $w=\id$ is obvious.
\item For $W_{s_\alpha}$, we start with \eqref{eq:a_W}. A change of variables $n_1 y_2 \mapsto n'_1$ gives
\ba
W_{s_\alpha} (g,\nu,\chi_{m_1,0}) &= y_1^{\nu_1+2} y_2^{\nu_1+1} \int_\R ({n'_1}^2+y_1^2)^{\nu_2-\nu_1-1/2} \e\rb{-\frac{m_1}{y_2} n'_1} dn'_1\\
&= y_1^{\nu_2+3/2} y_2^{\nu_1+1} \int_\R \rb{\frac{y_1}{{n'_1}^2+y_1^2}}^{\nu_2-\nu_1-1/2} \e\rb{-\frac{m_1}{y_2} n'_1} dn'_1\\
&= y_1^{\nu_2+3/2} y_2^{\nu_1+1} W(y_1, \nu_1-\nu_2, \chi_{m_1/y_2}).
\ea

\item For $W_{s_\beta}$, we start with \eqref{eq:b_W}, and rewrite
\ba
W_{s_\beta}(g,\nu,\chi_{0,m_2}) &= y_1^{\nu_1+2} y_2^{2\nu_2-\nu_1+1} \int_\R (y_2^4+n_5^2)^{\nu_1/2-\nu_2-1/2} \e(-m_2n_5) dn_5\\
&= y_1^{\nu_1+2} \int_\R \rb{\frac{y_2^2}{y_2^4+n_5^2}}^{\nu_2-\nu_1/2+1/2} \e(-m_2n_5) dn_5\\
&= y_1^{\nu_1+2} W\rb{y_2^2, \nu_2-\frac{\nu_1}{2}, \chi_{m_2}}.
\ea

\item For $W_{s_\alpha s_\beta}$, we start with \eqref{eq:ab_W}. A change of variables $n_4y_2 \mapsto n'_4$ gives
\ba
W_{s_\alpha s_\beta} (g,\nu,\chi_{0,m_2}) &= y_1^{\nu_1+2} y_2^{\nu_1+1} \int_\R \int_\R (y_2^4+n_5^2)^{\nu_1/2-\nu_2-1/2} ((y_2^4+n_5^2)y_1^2+{n'_4}^2)^{\nu_2-\nu_1-1/2}\\
&\hspace{8.1cm} \times \e(-m_2n_5) dn'_4 dn_5\\
&= y_1^{2\nu_2-\nu_1+2} y_2^{\nu_1+1} B\rb{\frac{1}{2}, \nu_1-\nu_2} \int_\R (y_2^4+n_5^2)^{-\nu_1/2-1/2}\e(-m_2n_5) dn_5\\
&= y_1^{2\nu_2-\nu_1+2} B\rb{\frac{1}{2}, \nu_1-\nu_2} \int_\R \rb{\frac{y_2^2}{y_2^4+n_5^2}}^{\nu_1/2+1/2}\e(-m_2n_5) dn_5\\
&= y_1^{2\nu_2-\nu_1+2} B\rb{\frac{1}{2}, \nu_1-\nu_2} W\rb{y_2^2, \frac{\nu_1}{2}, \chi_{m_2}}.
\ea

\item For $W_{s_\beta s_\alpha}$, we start with \eqref{eq:ba_W}, and rewrite
\begin{align*}
W_{s_\beta s_\alpha} (g,\nu,\chi_{m_1,0}) &= y_1^{\nu_1+2} y_2^{\nu_1+2} \int_\R \int_\R (n_1^2y_2^2+y_1^2)^{\nu_2-\nu_1-1/2} \\
&\hspace{2.4cm} \times (n_2^2+(n_1^2y_2^2+y_1^2)^2)^{\nu_1/2-\nu_2-1/2} \e(-m_1n_1) dn_1 dn_2\\
&= y_1^{\nu_1+2} y_2^{\nu_1+2} B\rb{\frac{1}{2}, \nu_2-\frac{\nu_1}{2}} \int_\R (n_1^2y_2^2+y_1^2)^{-\nu_2-1/2} \e(-m_1n_1) dn_1.
\end{align*}
A change of variables $n_1 y_2 \mapsto n'_1$ gives
\ba
&y_1^{\nu_1+2} y_2^{\nu_1+1} B\rb{\frac{1}{2}, \nu_2-\frac{\nu_1}{2}} \int_\R ({n'_1}^2+y_1^2)^{-\nu_2-1/2} \e\rb{-\frac{m_1}{y_2}n'_1} dn'_1\\
&= y_1^{\nu_1-\nu_2+3/2} y_2^{\nu_1+1} B\rb{\frac{1}{2}, \nu_2-\frac{\nu_1}{2}} \int_\R \rb{\frac{y_1}{{n'_1}^2+y_1^2}}^{\nu_2+1/2} \e\rb{-\frac{m_1}{y_2}n'_1} dn'_1\\
&= y_1^{\nu_1-\nu_2+3/2} y_2^{\nu_1+1} B\rb{\frac{1}{2}, \nu_2-\frac{\nu_1}{2}} W(y_1, \nu_2, \chi_{m_1/y_2}). 
\ea
\item For $W_{s_\alpha s_\beta s_\alpha}$, we start with \eqref{eq:aba_W}. A change of variables $n_2+n_1n_4 \mapsto n'_2$ gives
\ba
&W_{s_\alpha s_\beta s_\alpha}(g,\nu, \chi_{m_1,0}) = y_1^{\nu_1+2} y_2^{\nu_1+2} \int_\R \int_\R \int_\R \rb{(y_1^2+n_1^2y_2^2)^2 + {n'_2}^2}^{\nu_1/2-\nu_2-1/2}\\
&\hspace{0.5cm}\times(y_1^4y_2^2+{n'_2}^2y_2^2-2n_1n'_2n_4y_2^2+n_1^2n_4^2y_2^2+n_1^2y_1^2y_2^4+n_4^2y_1^2)^{\nu_2-\nu_1-1/2} \e(-m_1n_1) dn_1 dn'_2 dn_4.
\ea
Completing the square with respect to $n_4$ followed by a change of variables $n_4-\frac{n_1n'_2y_2^2}{(n_1^2y_2^2+y_1^2)} \mapsto n'_4$ gives
\ba
&y_1^{\nu_1+2} y_2^{\nu_1+2} \int_\R \int_\R \int_\R \rb{(y_1^2+n_1^2y_2^2)^2 + {n'_2}^2}^{\nu_1/2-\nu_2-1/2}(n_1^2y_2^2+y_1^2)^{\nu_2-\nu_1-1/2} \\
&\hspace{2.6cm} \times \rb{{n'_4}^2+\frac{y_1^2y_2^2((n_1^2y_2^2+y_1^2)^2+{n'_2}^2)}{(n_1^2y_2^2+y_1^2)^2}}^{\nu_2-\nu_1-1/2} \e(-m_1n_1) dn_1 dn'_2 dn'_4\\
&= y_1^{2\nu_2-\nu_1+2} y_2^{2\nu_2-\nu_1+2} B\rb{\frac{1}{2}, \nu_1-\nu_2} \int_\R \int_\R \rb{(y_1^2+n_1^2y_2^2)^2 + {n'_2}^2}^{-\nu_1/2-1/2}\\
&\hspace{7.2cm} \times (n_1^2y_2^2+y_1^2)^{\nu_1-\nu_2-1/2} \e(-m_1n_1) dn_1 dn'_2\\
&= y_1^{2\nu_2-\nu_1+2} y_2^{2\nu_2-\nu_1+2} B\rb{\frac{1}{2}, \frac{\nu_1}{2}} B\rb{\frac{1}{2}, \nu_1-\nu_2} \int_\R (n_1^2y_2^2+y_1^2)^{-\nu_2-1/2} \e(-m_1n_1) dn_1.
\ea
A change of variables $n_1 y_2 \mapsto n'_1$ then gives
\ba
&y_1^{2\nu_2-\nu_1+2} y_2^{2\nu_2-\nu_1+1} B\rb{\frac{1}{2}, \frac{\nu_1}{2}} B\rb{\frac{1}{2}, \nu_1-\nu_2} \int_\R ({n'_1}^2+y_1^2)^{-\nu_2-1/2} \e\rb{-\frac{m_1}{y_2}n'_1} dn'_1\\
&= y_1^{\nu_2-\nu_1+3/2} y_2^{2\nu_2-\nu_1+1} B\rb{\frac{1}{2}, \frac{\nu_1}{2}} B\rb{\frac{1}{2}, \nu_1-\nu_2} \int_\R \rb{\frac{y_1}{{n'_1}^2+y_1^2}}^{\nu_2+1/2} \e\rb{-\frac{m_1}{y_2}n'_1} dn'_1\\
&= y_1^{\nu_2-\nu_1+3/2} y_2^{2\nu_2-\nu_1+1} B\rb{\frac{1}{2}, \frac{\nu_1}{2}} B\rb{\frac{1}{2}, \nu_1-\nu_2} W(y_1, \nu_2, \chi_{m_1/y_2}). 
\ea

\item For $W_{s_\beta s_\alpha s_\beta}$, we start with \eqref{eq:bab_W}. A change of variables $n'_2 = n_2-n_1n_4$ gives
\ba
&W_{s_\beta s_\alpha s_\beta} (g,\nu,\chi_{0,m_2}) = y_1^{\nu_1+2} y_2^{\nu_1+2} \int_\R \int _\R \int_\R (y_1^2 y_2^4+n_5^2 y_1^2+n_4^2 y_2^2)^{\nu_2-\nu_1-1/2}\\
&\hspace{0.5cm}\times\rb{y_1^4 y_2^4+n_5^2 y_1^4 +2 n_4^2 y_1^2 y_2^2 +{n'_2}^2 y_2^4 +(n'_2n_5 - n_4^2)^2}^{\nu_1/2-\nu_2-1/2} \e(-m_2n_5) dn'_2 dn_4 dn_5.
\ea
Completing the square with respect to $n'_2$ followed by a change of variables $n'_2-\frac{n_4^2n_5}{y_2^4+n_5^2} \mapsto n''_2$ gives
\ba
&y_1^{\nu_1+2} y_2^{\nu_1+2} \int_\R \int _\R \int_\R (y_1^2 y_2^4+n_5^2 y_1^2+n_4^2 y_2^2)^{\nu_2-\nu_1-1/2}(y_2^4+n_5^2)^{\nu_1/2-\nu_2-1/2}\\
&\hspace{1.6cm}\times \rb{{n''_2}^2+\rb{\frac{y_1^2y_2^4+n_5^2y_1^2+n_4^2y_2^2}{y_2^4+n_5^2}}^2}^{\nu_1/2-\nu_2-1/2} \e(-m_2n_5)dn''_2 dn_4 dn_5\\
&= y_1^{\nu_1+2} y_2^{\nu_1+2} B\rb{\frac{1}{2}, \nu_2-\frac{\nu_1}{2}} \int_\R \int _\R (y_2^4+n_5^2)^{\nu_2-\nu_1/2-1/2}\\
&\hspace{5cm}\times(y_1^2 y_2^4+n_5^2 y_1^2+n_4^2 y_2^2)^{-\nu_2-1/2} \e(-m_2n_5) dn_4 dn_5.
\ea
A change of variables $n_4y_2 \mapsto n'_4$ then gives
\begin{align*}
&y_1^{\nu_1+2} y_2^{\nu_1+2} B\rb{\frac{1}{2}, \nu_2-\frac{\nu_1}{2}} \int_\R \int _\R (y_2^4+n_5^2)^{\nu_2-\nu_1/2-1/2}\\
&\hspace{5.2cm}\times(y_1^2 y_2^4+n_5^2 y_1^2+n_4^2 y_2^2)^{-\nu_2-1/2} \e(-m_2n_5) dn'_4 dn_5\\
&= y_1^{\nu_1-2\nu_2+2} y_2^{\nu_1+1} B\rb{\frac{1}{2}, \nu_2-\frac{\nu_1}{2}} B\rb{\frac{1}{2}, \nu_2} \int_\R (y_2^4+n_5^2)^{-\nu_1/2-1/2} \e(-m_2n_5) dn_5\\
&= y_1^{\nu_1-2\nu_2+2} B\rb{\frac{1}{2}, \nu_2-\frac{\nu_1}{2}} B\rb{\frac{1}{2}, \nu_2} \int_\R \rb{\frac{y_2^2}{y_2^4+n_5^2}}^{\nu_1/2+1/2} \e(-m_2n_5) dn_5\\
&= y_1^{\nu_1-2\nu_2+2} B\rb{\frac{1}{2}, \nu_2-\frac{\nu_1}{2}} B\rb{\frac{1}{2}, \nu_2} W\rb{y_2^2, \frac{\nu_1}{2}, \chi_{m_2}}.\qedhere
\end{align*}
\ee
\end{proof}

Let $\mc W(\nu,\chi)^{\operatorname{mod}}$ denote the subspace of functions with moderate growth in $\mc W(\nu,\chi)$. When $\chi$ is non-degenerate (i.e. $m_1m_2\ne 0$), we have the following celebrated theorem of Shalika and Wallach.

\begin{thm}[{Shalika \cite{Shalika1974}, Wallach \cite{Wallach1983}}]\label{multiplicity1}
 Let $\chi$ be a non-degenerate character on $U(\Z) \bs U(\R)$. Then $\mc W(\nu, \chi)^{\operatorname{mod}} \leq 1$. 
\end{thm}

\begin{rmk}
For a non-degenerate character $\chi$, we observe that $W_{w_0}(g,\nu,\chi) \in \mc W(\nu,\chi)^{\operatorname{mod}}$. By \Cref{multiplicity1}, we see that $W_{w_0}(g,\nu,\chi)$ is the unique function (up to a constant multiple) in $\mc W(\nu,\chi)^{\operatorname{mod}}$. This function is studied extensively by Ishii \cite{Ishii2005}.
\end{rmk}

\subsection{Fourier Coefficients}

Suppose that $\chi = \chi_{m_1,m_2}$ is a character of $U(\Z)\bs U(\R)$. Then the Fourier coefficient for the minimal Eisenstein series $E_0(g,\nu)$ corresponding to $\chi$ is given by (see \cite{Shahidi2010})
\ba
E_{0,\chi}(g,\nu) := \int_{U(\Z)\bs U(\R)} E_0(\eta g,\nu) \bar\chi(\eta) d\eta.
\ea

\begin{rmk}
In principle, one may consider the Fourier coefficients along other subgroups. For example, for Siegel modular forms, one usually considers the Fourier coefficients along the upper right block, which forms an abelian group. Here, we consider the Fourier coefficients along the unipotent part $U$ of $G$. These Fourier coefficients find applications for instance in the constructions of L-functions via the Langlands-Shahidi method \cite{Shahidi2010}. 
\end{rmk}

To compute the Fourier coefficients $E_{0,\chi}(g,\nu)$, we break down the expression via Bruhat decomposition, and express them in terms of Whittaker functions. We have
\ba
E_{0,\chi}(g, \nu) &= \sum\limits_{w\in W} \int_{U(\Z)\bs U(\R)} E_{0,w}(\eta g, \nu) \ol\chi(\eta) d\eta\\
&= \sum\limits_{w\in W} \sum\limits_{\gamma \in R_w} \sum\limits_{\delta \in \Gamma_w} \int_{U(\Z)\bs U(\R)} I_0(\gamma \delta \eta g, \nu) \ol\chi(\eta) d\eta\\
&= \sum\limits_{w\in W} \sum\limits_{\gamma \in R_w} \int_{\ol U_w(\Z) \bs \ol U_w(\R)} \int_{U_w(\R)} I_0(\gamma \eta \eta' g, \nu) \ol\chi(\eta \eta') d\eta d\eta'.
\ea
Let $\gamma = b_1 w t b_2$ be a Bruhat decomposition, with $b_1, b_2 \in U$, $t\in T$. We may assume that $b_2 \in U_w$. Then
\ba
E_{0,\chi}(g,\nu) = \sum\limits_{w\in W} \sum\limits_{\substack{\gamma \in R_w\\\gamma = b_1 w t b_2}} \int_{\ol U_w(\Z) \bs \ol U_w(\R)} \int_{U_w(\R)} I_0(b_1 w t b_2 \eta \eta' g, \nu) \ol\chi(\eta \eta') d\eta d\eta'.
\ea
The change of variables $b_2 \eta \mapsto \eta$ gives
\ba
E_{0,\chi}(g,\nu) = \sum\limits_{w\in W} \sum\limits_{\substack{\gamma \in R_w\\\gamma = b_1 w t b_2}} \chi(b_2) \int_{\ol U_w(\Z) \bs \ol U_w(\R)} \int_{U_w(\R)} I_0(w t \eta \eta' g, \nu) \ol\chi(\eta \eta') d\eta d\eta'.
\ea
Now observe that
\ba
I_0(wtg, \nu) = I_0((wtw^{-1}) wg, \nu) = I_0(wtw^{-1}, \nu) I_0(wg,\nu). 
\ea
So the Fourier coefficient becomes
\ba
E_{0,\chi}(g,\nu) &= \sum\limits_{w\in W} \sum\limits_{\substack{\gamma \in R_w\\\gamma = b_1 w t b_2}} \chi(b_2) I_0(wtw^{-1}, \nu) \int_{\ol U_w(\Z) \bs \ol U_w(\R)} \int_{U_w(\R)} I_0(w \eta \eta' g, \nu) \ol\chi(\eta \eta') d\eta d\eta'\\
&= \sum\limits_{w\in W} \sum\limits_{\substack{\gamma \in R_w\\\gamma = b_1 w t b_2}} \chi(b_2) I_0(wtw^{-1}, \nu) \int_{\ol U_w(\Z) \bs \ol U_w(\R)} W_w(\eta' g, \nu, \chi) \ol\chi(\eta') d\eta'.
\ea
Recall that $W_w(g,\nu,\chi) = 0$ unless $\chi$ is trivial on $\ol U_w(\R)$. If $\chi$ is trivial on $\ol U_w(\R)$, then it follows from the definition of a Whittaker function that $W_w(\eta' g, \nu, \chi) = W_w(g,\nu,\chi)$ for $\eta' \in \ol U_w(\R)$. So the Fourier coefficient becomes
\ba
\sum\limits_{w\in W} \sum\limits_{\substack{\gamma \in R_w\\\gamma = b_1 w t b_2}} \chi(b_2) I_0(wtw^{-1}, \nu) \int_{\ol U_w(\Z) \bs \ol U_w(\R)} W_w(g, \nu, \chi).
\ea

Hence, to obtain the Fourier coefficients of $E_0(g,\nu)$, it suffices to evaluate for $w\in W$ the sum
\ba
E_{0,\chi,w}(g,\nu) := \sum\limits_{\substack{\gamma \in R_w\\ \gamma = b_1wtb_2}} \chi(b_2) I_0(wtw^{-1}, \nu) W_w(g,\nu,\chi). 
\ea

\ben
\item For $w=\id$, we have $R_{\id} = \cb{I_4}$. So we immediately obtain
\ba
E_{0,\chi,\id}(g,\nu) = W_{\id}(g,\nu,\chi).
\ea

\item For $w=s_\alpha$, we use \eqref{eq:a_R} and compute for $\gamma = b_1wtb_2 \in R_{s_\alpha}$ with Plücker coordinates $v$ that
\ba
I_0(wtw^{-1},\nu) = v_4^{2\nu_2-2\nu_1-1},
\ea
and $\chi_{m_1,m_2}(b_2) = \e(-m_1v_3/v_4)$.
Hence
\ba
E_{0,\chi,s_\alpha}(g,\nu) &= \sum\limits_{v_4\geq 1} \sum\limits_{\substack{v_3\ppmod{v_4}\\ (v_3,v_4)=1}} v_4^{2\nu_2-2\nu_1-1} \e\rb{-\frac{m_1v_3}{v_4}} W_{s_\alpha}(g,\nu,\chi)\\
&=\sum\limits_{v_4\geq 1} v_4^{2\nu_2-2\nu_1-1} c_{v_4}(m_1) W_{s_\alpha}(g,\nu,\chi),
\ea
where
\ba
c_m(n) := \sum_{\substack{t=1 \\ (t,m)=1}}^m \e\rb{\frac{nt}{m}}
\ea
is the classical Ramanujan sum. Using the well-known identity \cite[Proposition 3.1.7]{Goldfeld2006}
\begin{align}\label{eq:rsdf} 
\sum\limits_{n\geq 1} c_n(m) n^{-k-1} = \begin{cases} \displaystyle \frac{\sigma_{-k}(m)}{\zeta(k+1)} & \text{ if } m \neq 0,\\
\displaystyle \frac{\zeta(k)}{\zeta(k+1)} & \text{ if } m=0,
\end{cases}
\end{align}
we conclude that
\ba
E_{0,\chi,s_\alpha}(g,\nu) = \begin{cases} \displaystyle \frac{\sigma_{2\nu_2-2\nu_1}(m_1)}{\zeta(2\nu_1-2\nu_2+1)} W_{s_\alpha}(g,\nu,\chi) & \text{ if } m_1\neq 0,\\
\displaystyle \frac{\zeta(2\nu_1-2\nu_2)}{\zeta(2\nu_1-2\nu_2+1)} W_{s_\alpha}(g,\nu,\chi) & \text{ if } m_1=0.\end{cases}
\ea

\item For $w = s_\beta$, we use \eqref{eq:b_R} and compute for $\gamma = b_1wtb_2 \in R_{s_\beta}$ with Plücker coordinates $v$ that
\ba
I_0(wtw^{-1},\nu) = v_{23}^{\nu_1-2\nu_2-1},
\ea
and $\chi_{m_1,m_2}(b_2) = \e(-m_2v_{34}/v_{23})$. Hence
\ba
E_{0,\chi,s_\beta}(g,\nu) &=\sum\limits_{v_{23}\geq 1} \sum\limits_{\substack{v_{34}\ppmod{v_{23}}\\ (v_{23}, v_{34}) = 1}} v_{23}^{\nu_1-2\nu_2-1} \e\rb{-\frac{m_2v_{34}}{v_{23}}} W_{s_\beta}(g,\nu,\chi)\\
&= \sum\limits_{v_{23}\geq 1} v_{23}^{\nu_1-2\nu_2-1} c_{v_{23}}(m_2) W_{s_\beta}(g,\nu,\chi).
\ea
By \eqref{eq:rsdf}, we obtain
\ba
E_{0,\chi,s_\beta}(g,\nu) = \begin{cases} \displaystyle \frac{\sigma_{\nu_1-2\nu_2}(m_2)}{\zeta(2\nu_2-\nu_1+1)} W_{s_\beta}(g,\nu,\chi) & \text{ if } m_2\neq 0,\\
\displaystyle \frac{\zeta(2\nu_2-\nu_1)}{\zeta(2\nu_2-\nu_1+1)} W_{s_\beta}(g,\nu,\chi) & \text{ if } m_2 = 0.
\end{cases}
\ea

\item For $w = s_\alpha s_\beta$, we use \eqref{eq:ab_R} and compute for $\gamma = b_1wtb_2 \in R_{s_\alpha s_\beta}$ with Plücker coordinates $v$ that
\ba
I_0(wtw^{-1}, \nu) = v_2^{2\nu_2-2\nu_1-1} v_{23}^{\nu_1-2\nu_2-1} = v_2^{-\nu_1-2} d^{2\nu_2-\nu_1+1},
\ea
where $d = (v_2,v_4)$, and $\chi_{m_1,m_2}(b_2) = \e(m_2v_4/v_2)$. Hence
\ba
E_{0,\chi,s_\alpha s_\beta}(g,\nu) = \sum\limits_{v_2\geq 1} \sum\limits_{\substack{v_3,v_4\ppmod{v_2}\\ (v_2,v_3,v_4) = 1}} v_2^{-\nu_1-2} d^{2\nu_2-\nu_1+1} \e\rb{\frac{m_2v_4}{v_2}} W_{s_\alpha s_\beta}(g,\nu,\chi).
\ea
Write $v_2 = dv'_2$, $v_4 = dv'_4$. Then the sum can be rewritten as
\ba
E_{0,\chi,s_\alpha s_\beta}(g,\nu) &= \sum\limits_{d\geq 1} d^{2\nu_2-2\nu_1-1} \sum\limits_{v'_2\geq 1} {v'_2}^{-\nu_1-2} \sum\limits_{\substack{v'_4\ppmod{v'_2}\\ (v'_2, v'_4)=1}} \e\rb{\frac{m_2v'_4}{v'_2}} \sum\limits_{\substack{v_3\ppmod{dv'_2}\\ (d,v_3)=1}} W_{s_\alpha s_\beta} (g,\nu,\chi)\\
&= \sum\limits_{d\geq 1} \varphi(d) d^{2\nu_2-2\nu_1-1} \sum\limits_{v'_2\geq 1} {v'_2}^{-\nu_1-1} c_{v'_2}(m_2) W_{s_\alpha s_\beta} (g,\nu,\chi),
\ea
where $\varphi$ stands for the Euler totient function. By \eqref{eq:rsdf}, we obtain
\ba
E_{0,\chi,s_\alpha s_\beta}(g,\nu) &= \begin{cases} \displaystyle \frac{\zeta(2\nu_1-2\nu_2)}{\zeta(2\nu_1-2\nu_2+1)} \frac{\sigma_{-\nu_1}(m_2)}{\zeta(\nu_1+1)} W_{s_\alpha s_\beta}(g,\nu,\chi) & \text{ if } m_2 \neq 0,\\
\displaystyle \frac{\zeta(2\nu_1-2\nu_2)}{\zeta(2\nu_1-2\nu_2+1)}\frac{\zeta(\nu_1)}{\zeta(\nu_1+1)} W_{s_\alpha s_\beta}(g,\nu,\chi) & \text{ if } m_2=0.
\end{cases}
\ea

\item For $w = s_\beta s_\alpha$, we use \eqref{eq:ba_R} and compute for $\gamma = b_1wtb_2 \in R_{s_\beta s_\alpha}$ with Plücker coordinates $v$ that
\ba
I_0(wtw^{-1},\nu) = v_4^{2\nu_2-2\nu_1-1} v_{14}^{\nu_1-2\nu_2-1} = v_{14}^{-\nu_1-2} d^{2\nu_1-2\nu_2+1},
\ea
where $d = (v_{14},v_{24})$, and $\chi_{m_1,m_2} (b_2) = \e(m_1v_{24}/v_{14})$. Hence
\ba
E_{0,\chi,s_\beta s_\alpha}(g,\nu) = \sum\limits_{v_{14}\geq 1} \sum\limits_{\substack{v_{24}\ppmod{v_{14}}\\ v_{14} \mid d^2}} \sum\limits_{\substack{v_{34}\ppmod{v_{14}}\\ (d^2/v_{14}, v_{34}) = 1}} v_{14}^{-\nu_1-2} d^{2\nu_1-2\nu_2+1} \e\rb{\frac{m_1v_{24}}{v_{14}}} W_{s_\beta s_\alpha}(g,\nu,\chi).
\ea
Write $v_{14} = dv'_{14}$, $v_{24} = dv'_{24}$, and $d' = d^2/v_{14}$. Recall that we have $d = v'_{14} d'$. Then we have $v_{14} = d' {v'}_{14}^2$, and the sum can be rewritten as
\ba
E_{0,\chi,s_\beta s_\alpha}(g,\nu) &= \sum\limits_{d'\geq 1} {d'}^{\nu_1-2\nu_2-1} \sum\limits_{v'_{14}\geq 1} {v'}_{14}^{-2\nu_2-3} \sum\limits_{\substack{v'_{24}\ppmod{v'_{14}}\\ (v'_{14}, v'_{24}) = 1}} \e\rb{\frac{m_1v'_{24}}{v'_{14}}} \sum\limits_{\substack{v_{34}\ppmod{d'{v'}_{14}^2}\\ (d', v_{34}) = 1}} W_{s_\beta s_\alpha}(g,\nu,\chi)\\
&=\sum\limits_{d'\geq 1} \varphi(d') {d'}^{\nu_1-2\nu_2-1} \sum\limits_{v'_{14}\geq 1} {v'}_{14}^{-2\nu_2-1} c_{v'_{14}}(m_1) W_{s_\beta s_\alpha}(g,\nu,\chi).
\ea
By \eqref{eq:rsdf}, we obtain that
\ba
E_{0,\chi,s_\beta s_\alpha}(g,\nu) &= \begin{cases} \displaystyle \frac{\zeta(2\nu_2-\nu_1)}{\zeta(2\nu_2-\nu_1+1)} \frac{\sigma_{-2\nu_2}(m_1)}{\zeta(2\nu_2+1)} W_{s_\beta s_\alpha}(g,\nu,\chi) & \text{ if } m_1\neq 0,\\
\displaystyle \frac{\zeta(2\nu_2-\nu_1)}{\zeta(2\nu_2-\nu_1+1)} \frac{\zeta(2\nu_2)}{\zeta(2\nu_2+1)} W_{s_\beta s_\alpha}(g,\nu,\chi) & \text{ if } m_1= 0.
\end{cases}
\ea

\item For $w=s_\alpha s_\beta s_\alpha$, we use \eqref{eq:aba_R} and compute for $\gamma = b_1wtb_2 \in R_{s_\alpha s_\beta s_\alpha}$ with Plücker coordinates $v$ that
\ba
I_0(wtw^{-1}, \nu) = v_1^{2\nu_2-2\nu_1-1} v_{14}^{\nu_1-2\nu_2-1} = v_1^{-2\nu_2-3} \delta^{2\nu_2-\nu_1+1},
\ea
where 
\ba
d &= (v_1, v_2), \quad & \delta &= (d^2, v_1v_3+v_2v_4),
\ea
and we have $\chi_{m_1,m_2}(b_2) = \e(m_1v_2/v_1)$. Hence
\ba
E_{0,\chi,s_\alpha s_\beta s_\alpha}(g,\nu) = \sum\limits_{v_1\geq 1} \sum\limits_{\substack{v_2,v_3,v_4\ppmod{v_1}\\ (v_1,v_2,v_3,v_4)=1}} v_1^{-2\nu_2-3} \delta^{2\nu_2-\nu_1+1} \e\rb{\frac{m_1v_2}{v_1}} W_{s_\alpha s_\beta s_\alpha}(g,\nu,\chi).
\ea
Write $v_1 = dv'_1$, $v_2 = dv'_2$. Since $d \mid \delta$, we may also write $\delta = d \delta'$. Note that $\delta'  = (d, v'_1v_3+v'_2v_4)$ divides $d$. Then the sum can be rewritten as
\begin{multline*}
E_{0,\chi,s_\alpha s_\beta s_\alpha}(g,\nu) = \sum\limits_{d\geq 1} d^{-\nu_1-2} \sum\limits_{v'_1\geq 1} {v'_1}^{-2\nu_2-3} \sum\limits_{\substack{v'_2\ppmod{v'_1}\\ (v'_1,v'_2)=1}} \e\rb{\frac{m_1v'_2}{v'_1}}\\
\times\sum\limits_{\substack{v_3,v_4\ppmod{dv'_1}\\ (d,v_3,v_4)=1}} {\delta'}^{2\nu_2-\nu_1+1} W_{s_\alpha s_\beta s_\alpha}(g,\nu,\chi).
\end{multline*}
For fixed $l \mid d$, we find the number of pairs $(v_3, v_4)$ modulo $d$ satisfying 
\ba
(d,v_3,v_4)=1 \quad \text{ and } \quad (d, v'_1v_3+v'_2v_4) = l.
\ea
We first observe that for every residue class $(v_3, v_4)$ modulo $d$, we can find representatives such that $0\leq v'_1 v_3 + v'_2 v_4 < d$. As $(v'_1, v'_2) = 1$, we can find $u_3, u_4\in\Z$ such that $v'_1 u_3 + v'_2 u_4 = 1$. Then for $0\leq n < d$, the equation
\begin{align}\label{eq:aba_delta}
v'_1v_3 + v'_2 v_4 \equiv n \pmod{d}
\end{align}
has $d$ distinct solutions, given by $(v_3, v_4) = (nu_3+kv'_2, nu_4-kv'_1)$ for $0\leq k < d$. A residue class $(v_3, v_4)$ modulo $d$ satisfies $(d, v'_1v_3+v'_2v_4) = l$ if and only if $l = (n,d)$. Let $0\leq n < d$ be such that $(n,d) = l$. Then the number of solutions to \eqref{eq:aba_delta} satisfying $(d,v_3,v_4) = 1$ is given by $d\varphi(l)/l$. Meanwhile, the number of integers $0\leq n < d$ with $(n,d)=l$ is given by $\varphi(d/l)$. Hence, there are in total $d \varphi(d/l) \varphi(l)/l$ solutions for $(v_3, v_4)$ modulo $d$ such that $(d, v'_1v_3+v'_2v_4) = l$. Therefore
\ba
E_{0,\chi,s_\alpha s_\beta s_\alpha}(g,\nu) = \sum\limits_{d\geq 1} d^{-\nu_1-1} \sum\limits_{v'_1\geq 1} {v'_1}^{-2\nu_2-3} c_{v'_1}(m_1) \sum\limits_{l\mid d} \varphi\rb{\frac{d}{l}} \varphi(l) l^{2\nu_2-\nu_1}  W_{s_\alpha s_\beta s_\alpha}(g,\nu,\chi).
\ea
Writing $d = d'l$ gives
\ba
E_{0,\chi,s_\alpha s_\beta s_\alpha}(g,\nu) = \sum\limits_{d\geq 1} \varphi(d') {d'}^{-\nu_1-1} \sum\limits_{l\geq 1} \varphi(l) l^{2\nu_2-2\nu_1-1} \sum\limits_{v'_1\geq 1} {v'_1}^{-2\nu_2-3} c_{v'_1}(m_1) W_{s_\alpha s_\beta s_\alpha}(g,\nu,\chi).
\ea
By \eqref{eq:rsdf}, we obtain
\ba
E_{0,\chi,s_\alpha s_\beta s_\alpha}(g,\nu) = \begin{cases} \displaystyle \frac{\zeta(\nu_1)}{\zeta(\nu_1+1)} \frac{\zeta(2\nu_1-2\nu_2)}{\zeta(2\nu_1-2\nu_2+1)} \frac{\sigma_{-2\nu_2}(m_1)}{\zeta(2\nu_2+1)} W_{s_\alpha s_\beta s_\alpha}(g,\nu,\chi) & \text{ if } m_1\neq 0,\\
\displaystyle \frac{\zeta(\nu_1)}{\zeta(\nu_1+1)} \frac{\zeta(2\nu_1-2\nu_2)}{\zeta(2\nu_1-2\nu_2+1)} \frac{\zeta(2\nu_2)}{\zeta(2\nu_2+1)} W_{s_\alpha s_\beta s_\alpha}(g,\nu,\chi) & \text{ if } m_1 = 0.
\end{cases}
\ea

\item For $w = s_\beta s_\alpha s_\beta$, we use \eqref{eq:bab_R} and compute for $\gamma = b_1wtb_2 \in R_{s_\beta s_\alpha s_\beta}$ with Plücker coordinates $v$ that
\ba
I_0(wtw^{-1},\nu) = v_2^{2\nu_2-2\nu_1-1} v_{12}^{\nu_1-2\nu_2-1} = v_{12}^{-\nu_1-2} d_0^{2\nu_1-2\nu_2+1},
\ea
where $d_0 = (v_{12}, v_{13}, v_{14})$, and $\chi_{m_1,m_2}(b_2) = \e(m_2v_{14}/v_{12})$. Hence
\ba
E_{0,\chi,s_\beta s_\alpha s_\beta}(g,\nu) = \sum\limits_{v_{12}\geq 1} \sum\limits_{\substack{v_{13}, v_{14}, v_{23}\ppmod{v_{12}}\\ \text{conditions (see \eqref{eq:bab_R})}}} v_{12}^{-\nu_1-2} d_0^{2\nu_1-2\nu_2+1} \e\rb{\frac{m_2v_{14}}{v_{12}}} W_{s_\beta s_\alpha s_\beta}(g,\nu,\chi). 
\ea
Writing $d_1 = (v_{12}, v_{14})$, $v_{12} = d_1v'_{12}$, $v_{14} = d_1v'_{14}$, $v_{13} = d_1k$, and $d' = d_1/d_0$, $t = d_0/d'$, we expand the conditions above and rewrite the sum in terms of $d'$, $t$ and $v'_{12}$:
\ba
E_{0,\chi,s_\beta s_\alpha s_\beta}(g,\nu) &= \sum\limits_{d' \geq 1} {d'}^{-2\nu_2-3} \sum\limits_{t\geq 1} t^{\nu_1-2\nu_2-1} \sum\limits_{v'_{12}\geq 1} {v'}_{12}^{-\nu_1-2} \sum\limits_{\substack{v'_{14}\ppmod{v'_{12}}\\ (v'_{12}, v'_{14})=1}} \e\rb{\frac{m_2v'_{14}}{v'_{12}}}\\
&\hspace{5.2cm}\sum\limits_{\substack{v'_{13}\ppmod{d'v'_{12}}\\ (d', v'_{13}) = 1}} \sum\limits_{\substack{v_{23}\ppmod{{d'}^2tv'_{12}}\\ v_{23} = a+rv'_{12} \\ (r,t) = 1}} W_{s_\beta s_\alpha s_\beta}(g,\nu,\chi)\\
&= \sum\limits_{d' \geq 1} \varphi(d') {d'}^{-2\nu_2-1} \sum\limits_{t\geq 1} \varphi(t)t^{\nu_1-2\nu_2-1} \sum\limits_{v'_{12}\geq 1} {v'}_{12}^{-\nu_1-1} c_{v'_{12}}(m_2) W_{s_\beta s_\alpha s_\beta}(g,\nu,\chi).
\ea
By \eqref{eq:rsdf}, we obtain
\ba
E_{0,\chi,s_\beta s_\alpha s_\beta}(g,\nu) = \begin{cases} \displaystyle \frac{\zeta(2\nu_2)}{\zeta(2\nu_2+1)} \frac{\zeta(2\nu_2-\nu_1)}{\zeta(2\nu_2-\nu_1+1)} \frac{\sigma_{-\nu_1}(m_2)}{\zeta(\nu_1+1)} W_{s_\beta s_\alpha s_\beta} (g,\nu,\chi) & \text{ if } m_2\neq 0,\\
\displaystyle \frac{\zeta(2\nu_2)}{\zeta(2\nu_2+1)} \frac{\zeta(2\nu_2-\nu_1)}{\zeta(2\nu_2-\nu_1+1)} \frac{\zeta(\nu_1)}{\zeta(\nu_1+1)} W_{s_\beta s_\alpha s_\beta} (g,\nu,\chi) & \text{ if } m_2=0.
\end{cases}
\ea
\item For $w = w_0$, we use \eqref{eq:w0_R} and compute for $\gamma = b_1wtb_2 \in R_{s_\alpha s_\beta s_\alpha}$ with Plücker coordinates $v$ that
\ba
I_0(wtw^{-1},\nu) = v_1^{2\nu_2-2\nu_1-1} v_{12}^{\nu_1-2\nu_2-1},
\ea
and $\chi_{m_1,m_2}(b_2) = \e(m_1v_2/v_1+m_2v_{14}/v_{12})$. Hence
\ba
E_{0,\chi,w_0}(g,\nu) = \sum\limits_{\gamma \in R_{w_0}} v_1^{2\nu_2-2\nu_1-1} v_{12}^{\nu_1-2\nu_2-1} \e\rb{\frac{m_1v_2}{v_1}+\frac{m_2v_{14}}{v_{12}}} W_{w_0}(g,\nu,\chi).
\ea
Note that this is actually a Dirichlet series of $\Sp(4)$ Ramanujan sums. Indeed,
\ba
E_{0,\chi,w_0}(g,\nu) = \sum\limits_{v_1,v_{12}\geq 1} R_{v_1, v_{12}}(m_1,m_2) v_1^{2\nu_2-2\nu_1-1} v_{12}^{\nu_1-2\nu_2-1} W_{w_0}(g,\nu,\chi),
\ea
where $R_{v_1,v_{12}}(m_1,m_2)$ is the $\Sp(4)$ Ramanujan sum defined in \eqref{eq:Sp4rsd}. By \Cref{prp:Sp4rs}, we obtain
\ba
&E_{0,\chi,w_0}(g,\nu)\\
&\hspace{0.5cm}= \begin{cases} \displaystyle \frac{\sigma_{-\nu_2,\nu_2-\nu_1}(m_1,m_2)}{\zeta(2\nu_1-2\nu_2+1)\zeta(2\nu_2-\nu_1+1)\zeta(\nu_1+1)\zeta(2\nu_2+1)} W_{w_0}(g,\nu,\chi) & \text{ if } m_1,m_2\neq 0,\\
\displaystyle \frac{\sigma_{2\nu_2-2\nu_1} (m_1)}{\zeta(2\nu_1-2\nu_2+1)}\frac{\zeta(2\nu_2-\nu_1)}{\zeta(2\nu_2-\nu_1+1)}\frac{\zeta(\nu_1)}{\zeta(\nu_1+1)}\frac{\zeta(2\nu_2)}{\zeta(2\nu_2+1)} W_{w_0}(g,\nu,\chi) & \text{ if } m_1\neq 0, m_2= 0,\\
\displaystyle \frac{\sigma_{\nu_1-2\nu_2} (m_2)}{\zeta(2\nu_2-\nu_1+1)}\frac{\zeta(2\nu_1-2\nu_2)}{\zeta(2\nu_1-2\nu_2+1)}\frac{\zeta(\nu_1)}{\zeta(\nu_1+1)}\frac{\zeta(2\nu_2)}{\zeta(2\nu_2+1)} W_{w_0}(g,\nu,\chi) & \text{ if } m_1= 0, m_2\neq 0,\\
\displaystyle \frac{\zeta(2\nu_1-2\nu_2)}{\zeta(2\nu_1-2\nu_2+1)}\frac{\zeta(2\nu_2-\nu_1)}{\zeta(2\nu_2-\nu_1+1)}\frac{\zeta(\nu_1)}{\zeta(\nu_1+1)}\frac{\zeta(2\nu_2)}{\zeta(2\nu_2+1)} W_{w_0}(g,\nu,\chi) & \text{ if } m_1=m_2= 0.
\end{cases}
\ea
\ee
\begin{proof}[Proof of \Cref{Fourier:min}]
The theorem follows by combining the terms $E_{0,\chi,w}$ for $w\in W$, using the computations above.
\end{proof}

\section{Residual eisenstein series} \label{section:degenerate}

In this section we consider the residual Eisenstein series $E_\alpha(g,\nu,1)$ and $E_\beta(g,\nu,1)$. We start with the following proposition, which is easy to verify.

\begin{prp}\label{Eisenstein3}
We have
\ba
E_\alpha(g,\nu, E(*,s)) &= E_0(g, (\nu+s,\nu)),\\
E_\beta(g,\nu, E(*,s)) &= E_0\rb{g, \rb{\nu, \frac{\nu}{2}+s}}. 
\ea
\end{prp}

By taking the residues, we obtain the residual Eisenstein series $E_\alpha(g, \nu, 1)$, $E_\beta(g, \nu, 1)$. Precisely, we have the following.

\begin{prp}\label{EisensteinResidue}
We have
\ba
\Res_{s=1/2} E_0(g, (\nu+s, \nu)) &= \frac{3}{\pi} E_\alpha(g, \nu, 1), &
\Res_{s=1/2} E_0\rb{g, \rb{\nu, \frac{\nu}{2}+s}} &= \frac{3}{\pi} E_\beta(g,\nu,1).
\ea
\end{prp} 
\begin{proof}
It is well-known (see \cite[Theorem 3.1.10]{Goldfeld2006}) that $E(z,s)$ has a pole at $s = \frac{1}{2}$ with residue $\frac{3}{\pi}$. Putting this back into \Cref{Eisenstein3} yields the result.
\end{proof}

\subsection{Constant terms}
By taking residues of the constant terms for $E_0(g,\nu)$, we get the constant terms for $E_\alpha(g,\nu,1)$ and $E_\beta(g,\nu,1)$. 

\begin{cor}\label{constterm:Siegelmin}
The constant term for $E_\alpha(g,\nu,1)$ along the minimal parabolic is given by
\ba
C_\alpha^0 (g,\nu,1) = C_{\alpha,\id}^0 (g,\nu,1) + C_{\alpha, s_\beta}^0 (g,\nu,1) + C_{\alpha, s_\beta s_\alpha}^0(g,\nu,1) + C_{\alpha, s_\beta s_\alpha s_\beta}^0 (g,\nu,1),
\ea
where
\ba
C_{\alpha, \id}^0 (g,\nu,1) &=  y_1^{\nu+3/2} y_2^{\nu+3/2},\\
C_{\alpha, s_\beta}^0  (g,\nu,1) &= \frac{\Lambda(\nu+\frac{1}{2})}{\Lambda(\nu+\frac{3}{2})} y_1^{\nu+3/2} y_2^{-\nu+1/2},\\
C_{\alpha, s_\beta s_\alpha}^0  (g,\nu,1) &= \frac{\Lambda(2\nu)}{\Lambda(2\nu+1)}\frac{\Lambda(\nu+\frac{1}{2})}{\Lambda(\nu+\frac{3}{2})} y_1^{-\nu+3/2} y_2^{\nu+1/2},\\
C_{\alpha, s_\beta s_\alpha s_\beta}^0 (g,\nu,1) &= \frac{\Lambda(\nu-\frac{1}{2})}{\Lambda(\nu+\frac{1}{2})}\frac{\Lambda(2\nu)}{\Lambda(2\nu+1)}\frac{\Lambda(\nu+\frac{1}{2})}{\Lambda(\nu+\frac{3}{2})} y_1^{-\nu+3/2} y_2^{-\nu+3/2}.
\ea
\end{cor}

\begin{cor}
The constant term for $E_\alpha(g,\nu,1)$ along the Siegel parabolic is given by
\ba
C_\alpha (g,\nu,1) = C_{\alpha,\id} (g,\nu,1) + C_{\alpha, s_\beta s_\alpha} (g,\nu,1) + C_{\alpha, s_\beta s_\alpha s_\beta} (g,\nu,1),
\ea
where
\ba
C_{\alpha,\id} (g,\nu,1) &= y_1^{\nu+3/2} y_2^{\nu+3/2},\\
C_{\alpha, s_\beta s_\alpha}  (g,\nu,1) &= \frac{\Lambda(\nu+\frac{1}{2})}{\Lambda(\nu+\frac{3}{2})} E\rb{-n_1+\frac{y_1}{y_2}i, \nu} y_1y_2,\\
C_{\alpha, s_\beta s_\alpha s_\beta} (g,\nu,1) &= \frac{\Lambda(\nu-\frac{1}{2})}{\Lambda(\nu+\frac{1}{2})}\frac{\Lambda(2\nu)}{\Lambda(2\nu+1)}\frac{\Lambda(\nu+\frac{1}{2})}{\Lambda(\nu+\frac{3}{2})} y_1^{-\nu+3/2} y_2^{-\nu+3/2}.
\ea
\end{cor}

\begin{cor}
The constant term for $E_\alpha(g,\nu,1)$ along the Jacobi parabolic is given by
\ba
C_\alpha^\beta(g,\nu,1) = C_{\alpha, s_\beta}^\beta  (g,\nu,1) + C_{\alpha, s_\beta s_\alpha s_\beta}^\beta (g,\nu,1),
\ea
where
\ba
C_{\alpha, s_\beta}^\beta  (g,\nu,1) &= E\rb{-n_5+y_2^2i, \frac{\nu}{2}+\frac{1}{4}} y_1^{\nu+3/2},\\
C_{\alpha, s_\beta s_\alpha s_\beta}^\beta (g,\nu,1) &= \frac{\Lambda(2\nu)}{\Lambda(2\nu+1)} \frac{\Lambda(\nu+\frac{1}{2})}{\Lambda(\nu+\frac{3}{2})} E\rb{-n_5+y_2^2i, \frac{\nu}{2}-\frac{1}{4}} y_1^{-\nu+3/2}.
\ea
\end{cor}

\begin{cor}
The constant term for $E_\beta(g,\nu,1)$ along the minimal parabolic is given by
\ba
C_\beta^0 (g,\nu,1) = C_{\beta,\id}^0(g,\nu,1) + C_{\beta, s_\alpha}^0(g,\nu,1) + C_{\beta, s_\alpha s_\beta}^0(g,\nu,1) + C_{\beta,s_\alpha s_\beta s_\alpha}^0(g,\nu,1),
\ea
where
\ba
C_{\beta, \id}^0 (g,\nu,1) &=  y_1^{\nu+2},\\
C_{\beta, s_\alpha}^0 (g,\nu,1)  &= \frac{\Lambda(\nu+1)}{\Lambda(\nu+2)} y_1 y_2^{\nu+1},\\
C_{\beta, s_\alpha s_\beta}^0 (g,\nu,1) &= \frac{\Lambda(\nu)}{\Lambda(\nu+1)} \frac{\Lambda(\nu+1)}{\Lambda(\nu+2)} y_1 y_2^{-\nu+1},\\
C_{\beta, s_\alpha s_\beta s_\alpha}^0 (g,\nu,1) &= \frac{\Lambda(\nu-1)}{\Lambda(\nu)}\frac{\Lambda(\nu)}{\Lambda(\nu+1)} \frac{\Lambda(\nu+1)}{\Lambda(\nu+2)} y_1^{-\nu+2}.
\ea
\end{cor}

\begin{cor}
The constant term for $E_\beta(g,\nu,1)$ along the Siegel parabolic is given by
\ba
C_\beta^\alpha(g,\nu,1) = C_{\beta, s_\beta}^\alpha (g,\nu,1) + C_{\beta, s_\beta s_\alpha s_\beta}^\alpha (g,\nu,1),
\ea
where
\ba
C_{\beta, s_\beta}^\alpha (g,\nu,1) &= E\rb{-n_1+\frac{y_1}{y_2}i, \frac{\nu+1}{2}} y_1^{\nu/2+1} y_2^{\nu/2+1},\\
C_{\beta, s_\beta s_\alpha s_\beta}^\alpha (g,\nu,1) &= \frac{\Lambda(\nu)}{\Lambda(\nu+1)} \frac{\Lambda(\nu+1)}{\Lambda(\nu+2)} E\rb{-n_1+\frac{y_1}{y_2}i, \frac{\nu-1}{2}} y_1^{-\nu/2+1} y_2^{-\nu/2+1}.
\ea
\end{cor}

\begin{cor}\label{constterm:NSNS}
The constant term for $E_\beta(g,\nu,1)$ along the Jacobi parabolic is given by
\ba
C_\beta(g,\nu,1) = C_{\beta,\id} (g,\nu,1) + C_{\beta, s_\alpha s_\beta} (g,\nu,1) + C_{\beta, s_\alpha s_\beta s_\alpha} (g,\nu,1),
\ea
where
\ba
C_{\beta, \id} (g,\nu,1) &=  y_1^{\nu+2},\\
C_{\beta, s_\alpha s_\beta}  (g,\nu,1) &= \frac{\Lambda(\nu+1)}{\Lambda(\nu+2)} E\rb{-n_5+y_2^2i, \frac{\nu}{2}} y_1,\\
C_{\beta, s_\alpha s_\beta s_\alpha} (g,\nu,1) &= \frac{\Lambda(\nu-1)}{\Lambda(\nu)} \frac{\Lambda(\nu)}{\Lambda(\nu+1)} \frac{\Lambda(\nu+1)}{\Lambda(\nu+2)} y_1^{-\nu+2}.
\ea
\end{cor}

\subsection{Fourier coefficients}
Likewise, taking the residues of the Fourier coefficients for $E_0(g,\nu)$ allows us to find the Fourier coefficients for $E_\alpha (g,\nu,1)$ and $E_\beta(g,\nu,1)$. 
\begin{cor}\label{Fourier:Siegel}
The Fourier coefficients for $E_\alpha(g,\nu,1)$ are given as follows. For $m_1=m_2=0$ we have
\ba
E_{\alpha, \chi_{0,0}} &= W_{s_\alpha}\rb{g, \rb{\nu+\frac{1}{2}, \nu}, \chi_{0,0}} + \frac{\zeta(\nu+\frac{1}{2})}{\zeta(\nu+\frac{3}{2})} W_{s_\alpha s_\beta} \rb{g, \rb{\nu+\frac{1}{2}, \nu}, \chi_{0,0}}\\
&\hspace{0.5cm}+ \frac{\zeta(2\nu)}{\zeta(2\nu+1)} \frac{\zeta(\nu+\frac{1}{2})}{\zeta(\nu+\frac{3}{2})} W_{s_\alpha s_\beta s_\alpha}  \rb{g, \rb{\nu+\frac{1}{2}, \nu}, \chi_{0,0}}\\
&\hspace{0.5cm}+ \frac{\zeta(\nu-\frac{1}{2})}{\zeta(\nu+\frac{1}{2})} \frac{\zeta(2\nu)}{\zeta(2\nu+1)} \frac{\zeta(\nu+\frac{1}{2})}{\zeta(\nu+\frac{3}{2})} W_{w_0} \rb{g,\rb{\nu+\frac{1}{2},\nu},\chi_{0,0}}.
\ea
For $m_1\ne 0$, $m_2=0$ we have
\ba
E_{\alpha, \chi_{m_1,0}} &= \frac{\sigma_{-2\nu}(m_1)}{\zeta(2\nu+1)} \frac{\zeta(\nu+\frac{1}{2})}{\zeta(\nu+\frac{3}{2})} W_{s_\alpha s_\beta s_\alpha} \rb{g,\rb{\nu+\frac{1}{2}, \nu},\chi_{m_1,0}}.
\ea
For $m_1=0$, $m_2\ne 0$ we have
\ba
E_{\alpha, \chi_{0,m_2}} &= \frac{\sigma_{-\nu-1/2}(m_2)}{\zeta(\nu+\frac{3}{2})} W_{s_\alpha s_\beta} \rb{g,\rb{\nu+\frac{1}{2}, \nu},\chi_{0,m_2}}\\
&\hspace{0.5cm}+ \frac{\sigma_{-\nu+1/2} (m_2)}{\zeta(\nu+\frac{1}{2})} \frac{\zeta(2\nu)}{\zeta(2\nu+1)} \frac{\zeta(\nu+\frac{1}{2})}{\zeta(\nu+\frac{3}{2})} W_{w_0} \rb{g,\rb{\nu+\frac{1}{2}, \nu},\chi_{0, m_2}}.
\ea
For $m_1,m_2\ne 0$ we have $E_{\alpha, \chi_{m_1,m_2}} = 0$.
\end{cor}

\begin{cor}\label{Fourier:NS}
The Fourier coefficients for $E_\beta(g,\nu,1)$ are given as follows. For $m_1=m_2=0$ we have
\ba
E_{\beta, \chi_{0,0}} &= W_{s_\beta} \rb{g, \rb{\nu, \frac{\nu+1}{2}}, \chi_{0,0}} + \frac{\zeta(\nu+1)}{\zeta(\nu+2)} W_{s_\beta s_\alpha} \rb{g, \rb{\nu, \frac{\nu+1}{2}}, \chi_{0,0}}\\
&\hspace{0.5cm}+ \frac{\zeta(\nu)}{\zeta(\nu+1)}\frac{\zeta(\nu+1)}{\zeta(\nu+2)} W_{s_\beta s_\alpha s_\beta} \rb{g, \rb{\nu, \frac{\nu+1}{2}}, \chi_{0,0}}\\
&\hspace{0.5cm}+\frac{\zeta(\nu-1)}{\zeta(\nu)} \frac{\zeta(\nu)}{\zeta(\nu+1)}\frac{\zeta(\nu+1)}{\zeta(\nu+2)} W_{w_0} \rb{g, \rb{\nu, \frac{\nu+1}{2}}, \chi_{0,0}}.
\ea
For $m_1\ne 0$, $m_2=0$ we have
\ba
E_{\beta, \chi_{m_1,0}} &= \frac{\sigma_{-\nu-1}(m_1)}{\zeta(\nu+2)} W_{s_\beta s_\alpha} \rb{g, \rb{\nu, \frac{\nu+1}{2}}, \chi_{m_1,0}}\\
&\hspace{0.5cm}+ \frac{\sigma_{-\nu+1}(m_1)}{\zeta(\nu)} \frac{\zeta(\nu)}{\zeta(\nu+1)}\frac{\zeta(\nu+1)}{\zeta(\nu+2)} W_{w_0} \rb{g, \rb{\nu, \frac{\nu+1}{2}}, \chi_{m_1,0}}.
\ea
For $m_1=0$, $m_2\ne 0$ we have
\ba
E_{\beta, \chi_{0,m_2}} = \frac{\sigma_{\nu}(m_2)}{\zeta(\nu+1)} \frac{\zeta(\nu+1)}{\zeta(\nu+2)} W_{s_\beta s_\alpha s_\beta} \rb{g, \rb{\nu, \frac{\nu+1}{2}}, \chi_{0,m_2}}.
\ea
For $m_1,m_2\ne 0$ we have $E_{\beta, \chi_{m_1,m_2}}= 0$.
\end{cor}

\begin{rmk}
Expressions in \Crefrange{constterm:Siegelmin}{Fourier:NS} can obviously be further simplified, but the current form has better structural consistency with expressions in \Cref{section:const_terms,section:Fourier_coeff}.
\end{rmk}


\end{document}